\documentclass[12pt]{amsart}
\usepackage{mycommands}
\newcommand{\FPdim}{\text{FPdim}}
\newcommand{\op}{\oplus}
\newcommand{\ot}{\otimes}
\newcommand{\dimhom}{\text{dim Hom}}
\usepackage{fancyhdr}
\usepackage{setspace}
\title{Pseudo-unitary non-self-dual fusion categories of rank 4}

\email{hannahlarson@college.harvard.edu}

\begin{document}

\maketitle

\begin{abstract}
A fusion category of rank $4$ has either four self-dual simple objects or exactly two self-dual simple objects. We study fusion categories of rank $4$ with exactly two self-dual simple objects, giving nearly a complete classification of those based rings that admit pseudo-unitary categorification. More precisely, we show that if $\mathcal{C}$ is such a fusion category, then its Grothendieck ring $K(\mathcal{C})$ must be one of seven based rings, six of which have known categorifications. In doing so, we classify all based rings associated with near-group categories of the group $\mathbb{Z}/3\mathbb{Z}$.
\end{abstract}

\section{Introduction}

A \textit{fusion category} is a semi-simple rigid tensor category with finitely many simple objects such that the unit object is simple. A complete classification of fusion categories is considered out of reach (it would include the classification of all finite groups), but classification results have been reached in small cases. One notion of a ``small" fusion category is a category with a small number of simple objects or \textit{rank}. The only fusion category of rank 1 is the category of vector spaces. Fusion categories of rank $2$ and rank $3$ were classified in \cite{ranktwo} and \cite{ostrik} respectively. Modular tensor categories (fusion categories with some additional structure) of rank 4 were classified in \cite{RSW}, and non-self-dual modular tensor categories of rank 5 were classified in \cite{rank5}. In a fusion category of rank $4$, either all four simple objects are self-dual, or exactly two simple objects are self-dual. Here, we study fusion categories of rank $4$ with exactly two self-dual simple objects.

One approach to the problem of classifying fusion categories is to study related objects called based rings. A \textit{based ring}, in the sense of \cite{Lu}, is a ring with a basis $\mathbf{1} = X_0, X_1, X_2, \ldots$ over $\zz$ such that the coefficients $N_{ij}^k$, known as the \textit{structure constants}, which appear in the decompositions $X_iX_j = \sum_{k}N_{ij}^kX_k$ are nonnegative integers, and there exists an involution $i \mapsto i^*$ such that $N_{ij}^0 = \delta_{ij^*}$. The following multiplication table describes a based ring of rank $4$ with basis elements $\mathbf{1}, X, Y$ and $Z$:

\begin{center}
\begin{table}[h]
\doublespacing
\begin{tabular}{c||c|c|c|c}
$\cdot$  & $\mathbf{1}$ & $X$ & $Y$ & $Z$  \\ 
\hline
\hline
$\mathbf{1}$ & $\mathbf{1}$ & $X$ & $Y$ & $Z$  \\ 
\hline
$X$ & $X$ & $Y + 2Z$ & $2X + 2Y + Z$ & $\mathbf{1} + 2Y$ \\ 
\hline
$Y$ & $Y$ & $2X + 2Y + Z$ & $1 + 2X + 4Y + 2Z$ & $X + 2Y + 2Z$ \\
\hline
$Z$ & $Z$ & $\mathbf{1} + 2Y$ & $X + 2Y + 2Z$ & $2X + Y$
\end{tabular}
\vspace{4pt}
\caption{A Based Ring of Rank 4}
\label{example}
\end{table}
\end{center}

\vspace{-\baselineskip}

\noindent
The Grothendieck ring of a fusion category is a based ring. For example, the condition that there exists an involution $i \mapsto i^*$ such that $N_{ij}^0 = \delta_{ij^*}$ is motivated by the duality in a fusion category: if $\mathcal{C}$ is a rigid category with duality functor $A \mapsto A^*$, then for any simple object $X$, the simple object $\mathbf{1}$ appears exactly once in the decomposition $X \otimes X^*$.

An isomorphism of a based ring with the Grothendieck ring of a fusion category is called a \textit{categorification} of the based ring. We say a based ring $K$ is \textit{pseudo-unitary categorifiable} if there exists a pseudo-unitary category $\mathcal{C}$ such that $K \simeq K(\mathcal{C})$. There are many restrictions that pseudo-unitary categorifiable based rings are known to satisfy. For example, by Theorem 2.21 of \cite{ostrik}, the \textit{formal codegrees} $f_0 \geq f_1 \geq \ldots f_{n-1}$ of a pseudo-unitary fusion category must satisfy $\sum{1/f_i^2} \leq (1/2)(1 + 1/f_0)$. Here, if the matrices for multiplication by basis elements of the based ring are $\mathbf{1} = X_0, X_1, \ldots, X_{n-1}$, the formal codegrees of the based ring are the eigenvalues of the matrix $M = 1 + X_1X_1^* +  X_2X_2^* + \ldots + X_{n-1}X_{n-1}^*$. For example, the formal codegrees of the based ring in table \ref{example}, are $36 + 20\sqrt{3}, 36 - 20\sqrt{3}, 8, 8$ which do not satisfy $\sum{1/f_i^2} \leq (1/2)(1 + 1/f_0)$. Hence, the based ring in table \ref{example} has no pseudo-unitary categorifications.

As the above example shows, not all based rings have categorifications. In fact, from known examples it seems that the majority of based rings are not categorifiable. In this paper, we show that of the based rings of rank four with two self-dual basis elements (based rings whose involution fixes exactly two elements), at most seven are pseudo-unitary categorifiable. Of these seven, six have known categories associated with them.

The main theorems of this paper show the following:
\begin{thm}
Let $K$ be a based ring of rank four with basis $\mathbf{1}, X, Y, Z$, where $X$ and $Z$ are duals of each other and $Y$ is self-dual. If $K$ is pseudo-unitary categorifiable then $K$ is one of the following:
\begin{enumerate}
\item the Grothendieck ring of the Tambara-Yamagami category associated with $\zz/3\zz$ \cite{T-Y}
\item the Grothendieck ring of the category of representations of the alternating group $A_4$
\item the Grothendieck ring of the Izumi-Xu category \cite{CMS}
\item \label{e=6} the based ring with multiplication given by
\begin{align*}
X^2 &= Z & Y^2 &= 1 + X + 6Y + Z & Z^2 &= X \\
XY &= YX = Y & YZ &= ZY = Y & XZ &= ZX = 1.
\end{align*}
\item the Grothendieck ring of the category of representations of $\zz/4\zz$
\item \label{c=1,2} the based ring with multiplication given by
\vspace{-.2in}
\begin{singlespace}
\begin{align*}
X^2 &= cX + Y + cZ & Y^2 &= \mathbf{1} & Z^2 &= cX + Y + cZ \\
XY &= YX = Z & YZ &= ZY = X & XZ &= ZX = \mathbf{1} + cX + cZ
\end{align*}
\end{singlespace}
\noindent
where $c = 1$ or $2$.
\end{enumerate}
\end{thm}

\begin{rem}
It was shown recently by Zhengwei Liu and Noah Snyder that the based ring in (\ref{e=6}) is categorifiable. It was pointed out by Scott Morrison that the based ring in (\ref{c=1,2}) with $c = 1$ is also categorifiable via (the even part of) the subfactor $\mathcal{S}'$ from \cite{LMP} Theorem 1. A recent paper of Bruillard shows that only (2) and (5) have associated categories with a ribbon structure \cite{B4}.
\end{rem}

The proofs of the main theorems rely on results bounding the minimum number of roots of unity required to write certain quadratic irrationalities as sums of roots of unity. For example, we prove the following in section \ref{rootssection}:

\begin{thm}
Let $a, b, c \in \zz$ with $c$ nonnegative and square free. If there exist $n$ roots of unity $\theta_i$ such that $\sum_{i = 0}^n\theta_i = a + b \sqrt{c}$, then $n \geq |b|\phi(2c)$.
\end{thm}

This paper is organized as follows. In section \ref{par}, we parameterize based rings of rank $4$ with two self-dual basis elements. In section \ref{red}, we find that if $K$ is pseudo-unitary categorifiable, then it belongs to one of two one-parameter families. One of these families is associated with the near-group categories of $\zz/3\zz$. We study each of these families separately in sections \ref{fam1sec} and \ref{fam2sec} by assuming there exists a category $\mathcal{C}$ with $K(\mathcal{C})$ in the family, computing information about the Drinfeld center $\mathcal{Z}(\mathcal{C})$, and taking traces of balance isomorphisms. This gives us identities for the twists that can only be satisfied when the free parameter is small. These identities involve sums of twists, which are roots of unity, equaling certain quadratic irrationalities. Thus, we need tight bounds on how many roots of unity it takes to write these quadratic irrationalities. We refer to these bounds in sections \ref{fam1sec} and \ref{fam2sec} but delay their proofs to section ~\ref{rootssection}.

\section{Explicit Parameterization of Based Rings} \label{par}
In this section, we parameterize based rings of rank four with two self-dual basis elements in terms of some diophantine equations and then solve these equations to reduce the number of parameters necessary to describe the based rings. 

Let $c, e, k, l, p, q$ be nonnegative integers subject to the conditions
\vspace{-12pt}
\begin{singlespace}
\begin{align}
kl + lc &= lp + kq \label{foo1} \\
kp + le + kc &= 2lq + k^2 \label{foo2} \\
l^2 + c^2 &= 1 + q^2 + p^2 \label{foo3} \\
l^2 + k^2 + q^2 &= 1 + 2pk + qe. \label{foo4}
\end{align}
\end{singlespace}
\noindent
Let $K(c, e, k, l, p, q)$ be the based ring with the basis $\mathbf{1}, X, Y, Z$ and multiplication given by
\vspace{-12pt}
\begin{singlespace}
\begin{align*}
X^2 &= pX + lY + cZ & XY &= YX = qX + kY + lZ \\
Y^2 &= 1 + kX + eY + kZ & YZ &= ZY = lX + kY + qZ \\
Z^2 &= cX + lY + pZ & XZ &= ZX = 1 + pX + qY + pZ.
\end{align*}
\end{singlespace}
\noindent
The following classifies based rings of rank four with exactly two self-dual basis elements.
\begin{prop}
Let $K$ be a based ring of rank $4$ with exactly two self-dual basis elements. Then $K = K(c, e, k, l, p, q)$ for some nonnegative integers $c, e, k, l, p, q$ satisfying \eqref{foo1}--\eqref{foo4}.
\end{prop}
\begin{proof}
We first observe that any based ring of rank $4$ is commutative. The complexification of $K$ is isomorphic to a direct sum of matrix algebras over the complex numbers. Because we are in rank $4$, the complexification is isomorphic to either two-by-two matrices over $\cc$ or $\cc \oplus \cc \oplus \cc \oplus \cc$. Since we have a non-trivial homomorphism into the complex numbers, namely sending each object to its dimension, we must be in the latter case, so $K$ is commutative. 

Working in the basis $1, X, Y, Z$, the general matrices for left multiplication in the commutative ring where $1$ and $Y$ are the only self-dual basis elements are
\vspace{-12pt}
\begin{singlespace}
\[M_X = \left(\begin{array}{cccc}
0 & 0 & 0 & 1 \\
1 & p & q & p \\
0 & l & k & q \\
0 & c & l & p
\end{array}\right), \quad M_Y = \left(\begin{array}{cccc}
0 & 0 & 1 & 0 \\
0 & q & k & l \\
1 & k & e & k \\
0 & l & k & q
\end{array}\right), \quad M_Z = \left(\begin{array}{cccc}
0 & 1 & 0 & 0 \\
0 & p & l & c \\
0 & q & k & l \\
1 & p & q & p
\end{array}\right).\]
\end{singlespace}

\noindent
The repeated identical columns come from commutativity, the zeros and ones of the first column are determined by $\mathbf{1}$ being a unit, and the zeros and ones of the first row are determined by duality. We know $M_Y$ must be symmetric because $Y$ is self-dual. Also, $M_X = M_Z^\dagger$ because $X$ is the dual of $Y$. Commutativity of $K$ implies the above matrices commute, which gives us the restrictions \eqref{foo1}--\eqref{foo4}.
\end{proof}

\noindent
We now reduce the number of parameters needed to describe these based rings from six to four by reformulating the restrictions \eqref{foo1}--\eqref{foo4}.
\begin{defi} 
Let $x, y, g, d$ be integers satisfying
\vspace{-12pt}
\begin{singlespace}
\begin{align}
yg + xd + y \equiv 0 \mod 2, \label{Rdivisibility}\\
dxy = g(2x^2 - y^2) + x^2 + 1, \label{theRequation}
\end{align}
\end{singlespace}
\noindent
\noindent
and $\frac{yg + xd + y}{2}, 2xg - yd + 2x, gy, gx, \frac{yg + xd - y}{2}, xg + x \geq 0$.
We define
\vspace{-12pt}
\begin{singlespace}
\[R(x, y, g, d) = K\left(\frac{yg + xd + y}{2},  2xg - yd + 2x, gy, gx, \frac{yg + xd - y}{2}, xg + x\right).\]
\end{singlespace}
\noindent
\end{defi}
\noindent
The reader can check that the condition $dxy = g(2x^2 - y^2) + x^2 + 1$ implies \eqref{foo1}--\eqref{foo4}. Note that if integers $x, y, g, d$ satisfy these restrictions, then $-x, -y, -g, -d$ satisfy them as well. Also note that $xg \geq 0$ and $yg \geq 0$ implies either both $x, y \leq 0$ or both $x, y \geq 0$.
\begin{prop} \label {allK=R}
Any based ring $K(c, e, k, l, p, q)$ is of the form $R(x, y, g, d)$ for integers $x, y, g, d$,
with $(x, y) = 1$.
\end{prop}
\begin{proof}
Given $K(c, e, k, l, p, q)$, there exist $g, x$, and $y$ such that $l = gx$ and $k = gy$ with $(x, y) = 1$. Rewriting \eqref{foo1} in terms of $g, x$, and $y$ we find $x(gy+c-p) = yq$ which implies $x \mid q$ and $y \mid c-p$. Let $q = ax$ and $c-p = by$. Then \eqref{foo1} gives $g + b = a$.
Similarly, \eqref{foo2} implies $x \mid p+c-gy$ and $y \mid 2q-e$. Let 
$p+c-gy = dx$ and $2q-e = fy$. Then \eqref{foo2} gives $d = f$. We write $c, e, k, l, p, q$ as
\vspace{-12pt}
\begin{singlespace}
\begin{align} 
c &= (dx + gy + by)/2 & e &= 2(g + b)x - dy & k &= gy \label{xstocs}\\
l &= gx & p &= (dx + gy - by)/2 & q &= (g + b)x \label{xstocs'}
\end{align}
\end{singlespace}
\noindent
Note that the equations for $c$ and $p$ imply the divisibility constraint \eqref{Rdivisibility}. Finally, \eqref{foo3} and \eqref{foo4} give us
$(gx)^2 - 1 = (g+b)^2x^2 - by(dx + gy)$, or equivalently
\begin{equation}
dbxy = gb(2x^2 - y^2) + b^2x^2 + 1. \label{withb}
\end{equation}
As every term of \eqref{withb} except $1$ has a factor of $b$, we have $b = \pm 1$. 
If $b = 1$, \eqref{withb} becomes \eqref{theRequation}, and 
$K(c, e, k, l, p, q) = R(x, y, g, d)$.
If $b = -1$, then \eqref{withb} becomes \eqref{theRequation} for $R(-x, -y, -g, -d)$, and
$K(c, e, k, l, p, q) = R(-x, -y, -g, -d)$.
\end{proof}

The following propositions give some additional information about the parameters $x, y$ and $g$. The first gives a useful restriction on the parity of $x$ and $y$.
\begin{prop} \label{xymod}
Let $x, y, g, d$ be integers satisfying \eqref{Rdivisibility} and \eqref{theRequation}. Then $x + y \equiv 1\mod 2$.
\end{prop}
\begin{proof}
We use case work to get the following implications. First,
\vspace{-12pt}
\begin{singlespace}
\[dx \ \text{and} \ y(g+1)\ \text{both odd} \ \ \Rightarrow \ \ d, x, y \ \text{odd and} \ g \  \text{even} \ \ \Rightarrow \ \ \text{LHS of \eqref{theRequation} is odd, RHS is even.}\]
\end{singlespace}
\noindent
Hence, $2 \mid dx$ and $2 \mid y(g-1)$. Next, 
\vspace{-12pt}
\begin{singlespace}
\[x \ \text{is odd} \quad \Rightarrow \quad d \ \text{is even} \quad \Rightarrow \quad 2 \mid gy^2 \quad \Rightarrow \quad 
\begin{cases}
\text{either} \ gx \ \text{is even} \quad &\Rightarrow \quad y \ \text{is even} \\
\text{or} \ g \ \text{is odd} \quad &\Rightarrow \quad y \ \text{is even}
\end{cases}
\]
\end{singlespace}
\noindent
And,
$x \ \text{is even} \ \Rightarrow \ gy^2 \ \text{is odd} \ \Rightarrow \ y \ \text{is odd}$. Therefore, $x + y \equiv 1\mod 2$.
\end{proof}

\begin{prop} \label{gnot0}
If $R(x, y, g, d) = K(c, e, k, l, p, q)$ for $c, e, k, l, p, q \in \zz_{\geq 0}$, then $g \neq 0$.
\end{prop}
\begin{proof}
Assume $g = 0$. Then $k = gy = 0$ and $l = gx = 0$. The restriction \eqref{foo4} gives $e = 0$, and $q = 1$.
Then \eqref{foo3} gives $c^2 = 2 + p^2$, which is a contradiction. Hence $g \neq 0$.
\end{proof}

\section{First Reductions} \label{red}
In this section we show that if a based ring $K(c, e, k, l, p, q)$ admits pseudo-unitary categorification, it belongs to one of two one-parameter families.
Recall the matrices $M_X$, $M_Y$, and $M_Z$ for left-multiplication by basis elements:
\vspace{-12pt}
\begin{singlespace}
\[M_X = \left(\begin{array}{cccc}
0 & 0 & 0 & 1 \\
1 & p & q & p \\
0 & l & k & q \\
0 & c & l & p
\end{array}\right), \quad M_Y = \left(\begin{array}{cccc}
0 & 0 & 1 & 0 \\
0 & q & k & l \\
1 & k & e & k \\
0 & l & k & q
\end{array}\right), \quad M_Z = \left(\begin{array}{cccc}
0 & 1 & 0 & 0 \\
0 & p & l & c \\
0 & q & k & l \\
1 & p & q & p
\end{array}\right)\]
\end{singlespace}
\noindent
If $A = 1 + M_Y^2 + 2M_XM_Z$, then the \textit{formal codegrees} $f_1 \geq f_2, f_3, f_4$ are the roots of the characteristic polynomial of A. By Theorem 2.21 of \cite{ostrik}, if $K$ is pseudo-unitary categorifiable, the formal codegrees satisfy $f_1, f_2, f_3, f_4 > 0$ and
\vspace{-12pt}
\begin{singlespace}
\begin{align}
\frac{1}{f_1^2} + \frac{1}{f_2^2} + \frac{1}{f_3^2} + \frac{1}{f_4^2} &\leq \frac{1}{2} \left(1 + \frac{1}{f_1}\right). \label{foobar}
\end{align}
In addition, by Proposition 2.10 of \cite{ostrik}, we have
\begin{align}
\frac{1}{f_1} + \frac{1}{f_2} + \frac{1}{f_3} + \frac{1}{f_4} &= 1. \label{sumofreciprocals}
\end{align}
\end{singlespace}
\noindent
We will see that the characteristic polynomial of $A$ always factors over $\qq$ as $(t - \gamma)^2 \cdot (t^2 - \alpha t + \beta)$ for integers $\gamma, \alpha, \beta$. The following lemmas greatly restrict the possibilities for $\gamma$, which we will translate into restrictions on the parameters $x$ and $y$.

\begin{lm} \label{gammalm}
Any polynomial of the form $(t - \gamma)^2 \cdot (t^2 - \alpha t + \beta)$, for integers $\gamma, \alpha$, and $\beta$ with roots 
$f_1 \geq f_2, f_3, f_4$ that satisfies \eqref{foobar}, \eqref{sumofreciprocals}, and $\beta \geq \gamma^2$,
must have $2 < \gamma < 8$.
\end{lm}
\begin{proof}
Expand $(t - \gamma)^2 \cdot (t^2 - \alpha t + \beta) = t^4 + (-\alpha - 2\gamma)t^3 + (\beta + 2\gamma\alpha + \gamma^2)t^2 + (-\gamma^2\alpha - 2\gamma\beta)t + \gamma^2\beta$.
Consequently,
\vspace{-12pt}
\begin{singlespace}
\[\frac{\gamma^2\alpha + 2\gamma\beta}{\gamma^2\beta} = \frac{1}{f_1} + \frac{1}{f_2} + \frac{1}{f_3} + \frac{1}{f_4} = 1 \quad \Rightarrow \quad \frac{-\alpha}{\beta} = \frac{2}{\gamma} - 1.\]
\end{singlespace}
\noindent
\vspace{-12pt}
\begin{singlespace}
\begin{align*}
\Rightarrow \frac{1}{f_1^2} + \frac{1}{f_2^2} + \frac{1}{f_3^2} + \frac{1}{f_4^2} &= \left(\frac{1}{f_1} + \frac{1}{f_2} + \frac{1}{f_3} + \frac{1}{f_4}\right)^2 - 2 \cdot \sum_{i \neq j} \frac{f_i f_j}{f_1 f_2 f_3 f_4} \\
&= 1^2 - 2 \cdot \frac{\beta + 2\gamma\alpha + \gamma^2}{\gamma^2 \beta} \\
&= 1 - \frac{2}{\gamma^2} + \frac{4}{\gamma} \cdot \frac{-\alpha}{\beta} - \frac{2}{\beta} \\
&= 1 - \frac{2}{\gamma^2} + \frac{4}{\gamma} \cdot \left(\frac{2}{\gamma} - 1\right) - \frac{2}{\beta} \\
&= 1 + \frac{6}{\gamma^2} - \frac{4}{\gamma} - \frac{2}{\beta}.
\end{align*}
\end{singlespace}
\noindent
Hence, $1 + 6/\gamma^2 - 4/\gamma - 2/\beta \leq (1/2) \cdot (1 + 1/f_1).$ Since $\beta \geq \gamma^2$, we have $f_1f_2f_3f_4 = \gamma^2\beta \geq \gamma^4$, so either there exists $i$ such that $f_i > \gamma$, or $f_i = \gamma$ for all $i$. In the latter case, $\gamma = 4$ because of condition \eqref{sumofreciprocals}. Otherwise, $f_1 > \gamma$ implies 
$\frac{1}{2} (1 + 1/f_1) < \frac{1}{2} (1 + 1/\gamma)$.
Using this, and $2/\beta \leq 2/\gamma^2$,
\vspace{-12pt}
\begin{singlespace}
\[1 + \frac{6}{\gamma^2} - \frac{4}{\gamma} - \frac{2}{\gamma^2} < \frac{1}{2} \left(1 + \frac{1}{\gamma}\right) \quad \Rightarrow\quad \gamma < 8.\]
\end{singlespace}
\noindent
Also, by condition \eqref{sumofreciprocals}, $\gamma > 2$.
\end{proof}
\begin{lm} \label{splittinglm}
If $f_1, f_2, f_3, f_4$ are nonnegative integers with $f_1 \geq f_2 \geq f_3 \geq f_4$ and $f_i = f_j$ for some $i \neq j$ satisfying \eqref{sumofreciprocals}
then $(f_1, f_2, f_3, f_4)$ is one of the following
\vspace{-12pt}
\begin{singlespace}
\begin{align*}
(12, 12&, 3, 2) & (8, 8&, 4, 2) & (10, 5&, 5, 2) & (6, 6&, 6, 2)\\
(6, 6&, 3, 3) & (6, 4&, 4, 3) & (12, 4&, 3, 3) & (4, 4&, 4, 4) 
\end{align*}
\end{singlespace}
\end{lm}
\begin{proof}
We begin by considering the case when $f_4 = 2$. Clearly, $f_3 \neq 2$. If $f_3 = 3$ then we must have $f_2 = f_1 = 12$; if $f_3 = 4$, then $f_2 \neq 4$ so we must have $f_2 = f_1 = 8$; if $f_3 = 5$, either $f_2 = 5 \Rightarrow f_1 = 10$, or $f_2 = f_1 \Rightarrow 2/f_2 = 1 - 1/2 - 1/5 = 3/10 \Rightarrow f_2 = 20/3$, a contradiction; if $f_3 = 6$, we must have $f_2 = f_3$ or $f_2 = f_1$, both of which imply $(f_1, f_2, f_3, f_4) = (6, 6, 6, 2)$. If $f_3 > 6$ then \eqref{sumofreciprocals} is not satisfied.

Now consider $f_4 = 3$. If $f_3 = 3$, then $1/f_2 + 1/f_1 = 1/3$. We see $f_2 = 4 \Rightarrow f_1 = 12$, $f_2 = 5 \Rightarrow f_1 = 2/15$, a contradiction, $f_2 = 6 \Rightarrow f_1 = 6$, and $f_2 > 6 \Rightarrow f_2 > f_1$ another contradiction. If $f_3 = 4$, then $f_2 = 4 \Rightarrow f_1 = 6$, and $f_2 = f_1 \Rightarrow 2/f_2 = 5/12 \Rightarrow f_2 = 5/6$, a contradiction. For $f_3 > 4$, \eqref{sumofreciprocals} fails.

Finally, if $f_4 = 4$, we must have $f_4 = f_3 = f_2 = f_1$. For $f_4 > 4$, there must exist some $f_i < 4$ for \eqref{sumofreciprocals} to be satisfied, which contradicts $f_1 \geq f_2 \geq f_3 \geq f_4$.
\end{proof}
\begin{thm} \label{2families}
Assume there exists a pseudo-unitary category $\mathcal{C}$ such that $K(\mathcal{C}) = K(c, e, k, l, p, q)$. Then either $K(\mathcal{C}) = K(c, 0, 0, 1, c, 0)$ or $K(\mathcal{C}) = K(1, e, 1, 0, 0, 0)$.
\end{thm}
\begin{proof}
Let $x, y, g, d$ be integers such that
$K(c, e, k, l, p, q) = R(x, y, g, d)$. Compute the characteristic polynomial $P_A(t)$ of $A = 1 + M_Y^2 + 2M_XM_Z$ in terms of $x, y, g$, and $d$. Let $\gamma = 2x^2 + y^2 + 2$. 
Evaluating $P_A(\gamma)$ and $P_A'(\gamma)$, and factoring over $\qq[x, y, g, d]$, we get something divisible by \eqref{theRequation}. Thus, $\gamma$ is a double root of $P_A(t)$, i.e.\ $P_A(t) = (t - \gamma)^2 \cdot (t^2 - \alpha t + \beta)$ for integers $\alpha$ and $\beta$. 

We claim $2 < \gamma < 8$. If $(t^2 - \alpha t + \beta)$ is irreducible, $\gamma^2 \mid \beta$ by Corollary 2.14 of \cite{ostrik}. In particular, $\gamma^2 \leq \beta$, so Lemma \ref{gammalm} implies $2 < \gamma < 8$. If $(t^2 - \alpha t + \beta)$ splits into two linear factors, then Lemma \ref{splittinglm} implies $\gamma = 3, 4, 5, 6, 8$, or $12$. The cases that remain to be considered are $\gamma = 12$, and $\gamma = 8$. Since $12$ cannot be written as $2x^2 + y^2 + 2$ for integers $x, y$, we can't have $\gamma = 12$.

If $\gamma = 8$, the other two roots are $2$ and $4$, by Lemma \ref{splittinglm}. The product of the four roots $8 \cdot 8 \cdot 4 \cdot 2 = 512$ equals $\det A$. As $\gamma = 8$, we have $(x, y) = (\pm 1, \pm 2)$. In both cases, $d = 1 - g$ by \eqref{theRequation} and $\det A = 4608g^2 + 1024g + 512$.
By Proposition \ref{gnot0}, $g \neq 0$, so $\det A > 512$.
Hence, $\gamma \neq 8$.

Since $\gamma = 2x^2 + y^2 + 2$ for integers $x, y$ satisfying
\vspace{-12pt}
\begin{singlespace}
\begin{align*}
(x, y) &= 1 & &\text{by Proposition \ref{allK=R}}\\
x + y &\equiv 1\mod 2 & &\text{by Lemma \ref{xymod},}
\end{align*}
\end{singlespace}
\noindent
either $(x, y) = (0, \pm 1)$ or $(x, y) = (\pm 1, 0)$. Using \eqref{withb} and $k, l \geq 0$, we can determine $g$ and $b$. Then using \eqref{xstocs} and \eqref{xstocs'}, we determine $c, e, k, l, p, q$. Note that $d$ is a free parameter, which gives the based rings $K(1, e, 1, 0, 0, 0)$ when $(x, y) = (0, \pm 1)$, and $K(c, 0, 0, 1, c, 0)$ when $(x, y) = (\pm 1, 0)$.
\end{proof}

\begin{rem} \label{fc=3}
If $K(1, e, 1, 0, 0, 0) = R(x, y, g, d)$, then $(x, y) = (0, \pm 1)$ and $P_A(t)$ has a repeated root equal to $3$, so two of the formal codegrees of $K(1, e, 1, 0, 0, 0)$ are $f_1 = f_2 = 3$.
\end{rem} 
\begin{rem} \label{fc=4}
If $K(c, 0, 0, 1, c, 0) = R(x, y, g, d)$, then $(x, y) = (\pm 1, 0)$ and $P_A(t)$ has a repeated root equal to $4$, so two of the formal codegrees of $K(c, 0, 0, 1, c, 0)$ are $f_1 = f_2 = 4$.
\end{rem}

For the remainder of this paper, we will write $K_1(e)$ for $K(1, e, 1, 0, 0, 0)$ and $K_2(c)$ for $K(c, 0, 0, 1, c, 0)$. The remainder of the paper shows that $K_1(e)$ can only admit categorification when $e = 0, 2, 3, 6$, and $K_2(c)$ can only admit categorification when $c = 0, 1, 2$.

\section{Categorifications of $K_1(e)$} \label{fam1sec}
Throughout this section we will assume that $\mathcal{C}$ is a fusion category with $K(\mathcal{C}) \simeq K_1(e)$. We study the Drinfeld center $\mathcal{Z}(\mathcal{C})$, which is a modular tensor category, to arrive at a contradiction when $e$ is not equal to $0, 2, 3, 6$. These values for $e$ give rise to the based rings (1) - (4) respectively of Theorem 1.1. 
Recall that the multiplication of basis elements in $K_1(e)$ is given by
\vspace{-12pt}
\begin{singlespace}
\begin{align*}
X^2 &= Z & Y^2 &= 1 + X + eY + Z & Z^2 &= X \\
XY &= YX = Y & YZ &= ZY = Y & XZ &= ZX = 1.
\end{align*}
\end{singlespace}
\noindent

\begin{rem}
$K_1(e)$ is the family of near-group categories associated with the group $\zz/3\zz$.
\end{rem}

\subsection{Induction and Forgetful Functors}
Let $F: \mathcal{Z}(\mathcal{C}) \to \mathcal{C}$ be the forgetful functor, and let the induction functor $I: \mathcal{C} \to \mathcal{Z}(\mathcal{C})$ be its right adjoint. We will study the effect of $I$ and $F$ on $K(\mathcal{C})$ and $K(\mathcal{Z}(\mathcal{C}))$. Propositions \ref{simpleobjects} and \ref{Yprop} describe the simple objects in $\mathcal{Z}(\mathcal{C})$, giving their images under $F$ and their dimensions.

Let $k = \frac{e}{3}$ and $\delta = \frac{e + \sqrt{e^2 + 12}}{2} = \frac{3k + \sqrt{9k^2 + 12}}{2}$.

\begin{prop} \label{dims}
The dimensions of simple objects in $\mathcal{C}$ are $\text{FPdim}(Y) = \delta$, and
\vspace{-12pt}
\begin{singlespace}
\[\text{FPdim}(\mathbf{1}) = \text{FPdim}(X) = \text{FPdim}(Z) = 1.\]
\end{singlespace}
\noindent
In particular, $\dim(\mathcal{C}) = 6 + 3k\delta$.
\end{prop}
\begin{proof}
Recall the matrix for left multiplication
\vspace{-12pt}
\begin{singlespace}
\[M_Y = \left(\begin{array}{cccc}
0 & 0 & 1 & 0 \\
0 & 0 & 1 & 0 \\
1 & 1 & e & 1 \\
0 & 0 & 1 & 0\
\end{array}\right).\]
\end{singlespace}
\noindent
The characteristic polynomial of $M_Y$ is $t^2 \cdot (t^2 - et - 3)$. Thus, $\FPdim(Y) = \frac{e + \sqrt{e^2 + 4 \cdot 3}}{2} = \delta$.

By the same method $\text{FPdim}(\mathbf{1}) = \text{FPdim}(X) = \text{FPdim}(Z) = 1$. We can now compute $\dim(\mathcal{C}) = \text{FPdim}(\mathbf{1})^2 + \text{FPdim}(X)^2 + \text{FPdim}(Y)^2 + \text{FPdim}(Z)^2 = 1 + 1 + \delta^2 + 1 = 6 + 3k\delta$.
\end{proof}

Note that $\delta$ is irrational unless $e = 2$, in which case $K_1(e)$ is the Grothendieck ring of the category of representations of the alternating group $A_4$, and so is already known to admit categorification. From now on, we will assume $e \neq 2$, and hence $\delta$ is irrational.

\begin{prop} \label{simpleobjects}
There exist non-isomorphic simble objects in $\mathbf{1}, A, B, C, D, E, G, H, J$ in $\mathcal{Z}(\mathcal{C})$ for which 
\vspace{-12pt}
\begin{singlespace}
\begin{align*}
I(\mathbf{1}) &= \mathbf{1} \op A \op B \op C &  I(X) &= B \op D \op E \op G, & I(Z) &= C \op D \op H \op J, \\[2pt]
\intertext{and} \\[-16pt]
F(A) &= \mathbf{1} \op kY & F(D) &= X \op \alpha Y \op Z & F(H) &= rY \op Z \\
F(B) &= \mathbf{1} \op X \op kY & F(E) &= X \op rY & F(J) &= pY \op Z \\
F(C) &= \mathbf{1} \op kY \op Z & F(G) &= X \op pY, \\[3pt]
\intertext{where $r, p$ and $\alpha$ are integers satisfying $\alpha + r + p = 2k$. In particular,} \\[-15pt]
\dim(A) &= 1 + k\delta & \dim(D) &= 2 + \alpha \delta & \dim(H) &= 1 + r\delta \\
\dim(B) &= 2 + k\delta & \dim(E) &= 1 + r\delta & \dim(J) &= 1 + p\delta \\
\dim(C) &= 2 + k\delta & \dim(G) &= 1 + p\delta.
\end{align*}
\end{singlespace}
\end{prop}
\begin{proof}
By Theorem 2.13 of \cite{ostrik}, since $K_1(e)$ has $4$ irreducible representations, the object $I(\mathbf{1}) \in \mathcal{Z}(\mathcal{C})$ decomposes into the sum of $4$ simple objects. Because
\vspace{-12pt}
\begin{singlespace}
\[\dim\hom(\mathbf{1}, I(\mathbf{1})) = \dim\hom(F(\mathbf{1}), \mathbf{1})) = \dim\hom(\mathbf{1}, \mathbf{1}) = 1,\]
\end{singlespace}
\noindent
one of these objects must be $\mathbf{1}$.
So let 
$I(\mathbf{1}) = \mathbf{1} \op A \op B \op C$.
Two of the formal codegrees of the based ring $K_1(e)$ are $f_1 = f_2 = 3$ (see Remark \ref{fc=3}), so by Theorem 2.13 of \cite{ostrik}, we can assume that $\dim(B) = \dim(C) = \frac{\dim(\mathcal{C})}{3} = 2 + k\delta$, whence $\FPdim(A)= \frac{\dim(\mathcal{C})}{3} - 1 = 1 + k\delta$. Consider $F(I(\mathbf{1}))$. We denote by $\mathbb{O}(\mathcal{C})$ the set of isomorphism classes of simple objects of $\mathcal{C}$. By Proposition 5.4 of \cite{ENO} we have
\vspace{-12pt}
\begin{singlespace}
\[F(\mathbf{1}) \op F(A) \op F(B) \op F(C) = F(I(\mathbf{1})) = \bigoplus_{X \in \mathbb{O}(\mathcal{C})} X \otimes \mathbf{1} \otimes X^* = 4 \cdot \mathbf{1} \op X \op eY \op Z.\]
\end{singlespace}
\noindent
Let $F(A) = \mathbf{1} \op lX \op mY \op nZ$. Then,
$1 + k \delta = \dim(A) = 1 + l + m\delta + n$.
Since $\delta \notin \qq$, we have $m = k$ and $l + n = 0$, so $l = n = 0$. Thus, $F(A) = \mathbf{1} \oplus kY$. We repeat the same arguments to determine $F(B)$ and $F(C)$.

\noindent
Next we consider $I(X)$. Using Proposition 5.4 of \cite{ENO}, we calculate
\vspace{-12pt}
\begin{singlespace}
\[F(I(X)) = \bigoplus_{Y \in \mathbb{O}(\mathcal{C})} Y \otimes X \otimes Y^* = \mathbf{1} \op 4X \op 3kY \op Z.\]
\end{singlespace}
\noindent
Since $\text{dim Hom}(I(\mathbf{1}), I(X)) = \text{dim Hom}(F(I(\mathbf{1}), X) = 1$, the decompositions of $I(X)$ and $I(\mathbf{1})$ have exactly one
simple object in common. As
\vspace{-12pt}
\begin{singlespace}
\[\text{dim Hom}(B, I(X)) = \text{dim Hom}(F(B), X) = 1,\]
\end{singlespace}
\noindent
this object is $B$.
Finally, as $\text{dim Hom}(I(X), I(X)) = \text{dim Hom}(F(I(X)), X) = 4$,
we see $I(X)$ is the sum of four simple objects.
Thus, let
$I(X) = B \op D \op E \op G$.
We have
\vspace{-12pt}
\begin{singlespace}
\begin{align*}
\left(\mathbf{1} \op X \op kY\right) \op F(D) \op F(E) \op F(G) &= F(I(X)) = \mathbf{1} \op 4X \op 3kY \op Z \\
\Rightarrow F(D) \op F(E) \op F(G) &= 3X \op 2kY \op Z.
\end{align*}
\end{singlespace}
\noindent
Since $\text{dim Hom}(D, I(X)) = \text{dim Hom}(F(D), X) = 1$, let $F(D) = X \op \alpha Y \op \beta Z$. By a similar argument, $F(E) = X \op rY \op sZ$ and $F(G) = X \op pY \op qZ$, where $\alpha + r + p = 2k$ and $\beta + s + q = 1$. Without loss of generality, we can assume $\beta = 1$, so $s = q = 0$.

Finally,
since dim Hom$(I(X), I(Z)) = \text{dim Hom}(F(I(X)), Z) = 1$, the decompositions of $I(Z)$ and $I(X)$ have exactly one simple object in common. By duality, there are objects $H = E^*$ and $J = G^*$ for which
$I(Z) = C \op D \op H \op J$,
and $F(H) = rY \op Z$ and $F(J) = pY \op Z$.
\end{proof}

Note that, under our assumption that $e \neq 2$, we have shown $3 \mid e$, i.e.\ $k$ is an integer.

\begin{prop} \label{Yprop}
There exist simple objects $L_i \in \mathcal{Z}(\mathcal{C})$ and positive integers $\gamma_i$ such that
\vspace{-12pt}
\begin{singlespace}
\[I(Y) = kA + kB + kC + kD + rE +pG + rH + pJ + \sum{\gamma_i L_i}\]
\end{singlespace}
\noindent
where the objects $L_i$ satisfy $F(L_i) = \gamma_iY$ and $\sum{\gamma_i^2 = 6 + 5k^2 -2(r^2 + p^2)}$.
\end{prop}
\begin{proof}
Consider $F(I(Y))$. We have
\vspace{-12pt}
\begin{singlespace}
\[F(I(Y)) = \bigoplus_{X \in \mathbb{O}(\mathcal{C})} X \otimes Y \otimes X^*
= 3k \op 3kX \op (6 + 9k^2)Y \op 3kZ.\]
\end{singlespace}
\noindent
Since $I$ and $F$ are adjoint, the multiplicity to which an object $O$ appears in the decomposition of $I(Y)$ equals the multiplicity of $Y$ in $F(O)$. Thus by Proposition \ref{simpleobjects},
\begin{equation}
I(Y) = kA + kB + kC + kD + rE +pG + rH + pJ + \sum{\gamma_i L_i} \label{I(Y)}
\end{equation}
where the objects $L_i$ satisfy $F(L_i) = \gamma_i Y$. Then,
\vspace{-12pt}
\begin{singlespace}
\begin{align*}
k^2 +k^2 + k^2 + k^2 + 2r^2 + 2p^2 + \sum{\gamma_i^2} &= \text{dim Hom}(F(I(Y)), Y) = 9k^2 + 6 \\
\Rightarrow \sum{\gamma_i^2} &= 6 + 5k^2 -2(r^2 + p^2). && \qedhere
\end{align*}
\end{singlespace}
\noindent
\end{proof}

\subsection{Twists}
We now compute the twists $\theta_X: X \to X$ of the simple objects in $\mathcal{Z}(\mathcal{C})$.
\begin{prop} \label{thetaABC}
The twists satisfy $\theta_A = \theta_B = \theta_C = 1$.
\end{prop}
\begin{proof}
Recall the dimensions of the simple objects $A, B$, and $C$ (see Proposition \ref{simpleobjects}). By Theorem 2.5 of \cite{ostrik},
\vspace{-12pt}
\begin{singlespace}
\[6 + 3k\delta = \dim(\mathcal{C}) = \tr(\theta_{I(\mathbf{1})})
= 1 + \left(1 + k\delta\right) \theta_A + \left(2 + k\delta\right) \theta_B + \left(2 + k\delta\right) \theta_C\]
\end{singlespace}
\noindent
so we must have $\theta_A = \theta_B = \theta_C = 1$.
\end{proof}

\noindent
Recall $\delta = \frac{3k + \sqrt{9k^2 + 12}}{2}$.
\begin{lm} \label{clm}
If $c$ is the squarefree part of $9k^2 + 12$, so that $\delta \in \qq(\sqrt{c})$, then $3 \mid c$ and $(c, 10) = 1$.
\end{lm}
\begin{proof}
First, $9k^2 + 12 = 3(3k^2 + 4)$, so $3 \, \, || \, \, 9k^2 + 12$, giving $3 \mid c$. Next, since squares are never $3$ mod $5$, we have $5 \nmid 9k^2 + 12 = (3k)^2 + 12$. Finally, we have $9k^2 + 12 \equiv 5, 12, 13, 16, 21, 28, 29$ mod $32$, so $9k^2 + 12$ is always divisible by an even power of $2$, and hence $2 \nmid c$.
\end{proof}

\noindent
Let $\omega$ be a root of the polynomial $x^2 + x + 1$.

\begin{prop} \label{rootsofunity}
The twists satisfy $\theta_D = \omega$ and $\theta_E = \theta_G = \theta_H = \theta_J = \omega^2$.
\end{prop}
\begin{proof}
By \cite{NgS}, we have
\vspace{-12pt}
\begin{singlespace}
\[\zeta\dim(\mathcal{C}) = \tr(\theta_{I(X)}^3) = \left(2 +k\delta\right) + (2 + \alpha \delta)\theta_D^3 + (1 + r\delta)\theta_E^3 + (1 + p\delta)\theta_G^3. \]
\end{singlespace}
\noindent
for some root of unity $\zeta$. Taking the absolute value of both sides implies $\theta_D^3 = \theta_E^3 = \theta_G^3 = 1$. Also, Theorem 2.5 of \cite{ostrik} says that
\vspace{-12pt}
\begin{singlespace}
\[0 = \tr(\theta_{I(X)}) = \left(2 + k\delta\right) + (2 + \alpha \delta)\theta_D + (1 + r\delta)\theta_E + (1 + p\delta)\theta_G, \]
\end{singlespace}
\noindent
and since $\delta \notin \qq(\omega)$ this implies
\vspace{-12pt}
\begin{singlespace}
\[0 = 2 + 2\theta_D + \theta_E + \theta_G \qquad \text{and} \qquad 0 = k + \alpha\theta_D + r\theta_E + p\theta_G.\]
\end{singlespace}
\noindent
Since $\theta_D \neq 1$ (otherwise the above equation would be strictly greater than $0$), we have $\theta_D = \omega$ which implies $\theta_E = \theta_G = \omega^2$.
\end{proof}

\begin{cor} \label{r+p=krem}
We have $\alpha = k$ and $r + p = k$.
\end{cor}
\begin{proof}
We have $0 = k + \alpha\omega + r\omega^2 + p\omega^2$. Setting the imaginary part equal to $0$ implies $\alpha = r + p$. Then, Proposition \ref{simpleobjects} gives $2k = \alpha + r + p = 2\alpha$, so we have $\alpha = k$ and $r + p = k$.
\end{proof}

\subsection{Categorifications of $K_1(e)$}
Recall the dimensions of $A, B, C, D,$ $E, G, H, J \in \mathcal{Z}(\mathcal{C})$ from Proposition \ref{simpleobjects}, and the decomposition of $I(Y)$ from Proposition \ref{Yprop}. Let $\theta_i = \theta_{L_i}$. We will take the trace of $\theta_{I(Y)}$ and $\theta_{I(Y)}$ and use the results in \cite{NgS} to arrive at expressions for the sums of the twists $\sum \gamma_i\theta_i$ and the sum of their squares $\sum \gamma_i^2\theta_i^2$ in terms of the parameter $k$. By Theorem 2.5 of \cite{ostrik}, we have
\vspace{-12pt}
\begin{singlespace}
\begin{align*}
0 = \tr(\theta_{I(Y)}) &= k(1 + k\delta) + k(2 + k\delta) + k(2 + k\delta) + k(2 + k\delta)\omega \\
& \qquad + 2r(1 + r\delta)\omega^2 + 2p(1 + p\delta) \omega^2 + \sum{\gamma_i^2 \theta_i \delta} \\
\Rightarrow \quad \sum{\gamma_i^2 \theta_i \delta} &= -k(1 + k\delta) - k(2 + k\delta) - k(2 + k\delta) - k(2 + k\delta)\omega \\
& \qquad - 2r(1 + r\delta)\omega^2 - 2p(1 + p\delta) \omega^2 \\
\sum{\gamma_i^2 \theta_i} &= \frac{-k}{2}\sqrt{9k^2 + 12} - 2rp \pm \frac{\sqrt{3}i}{2}(k^2 - 4rp).
\end{align*}
\end{singlespace}
\noindent
Next, \cite{NgS} says that if a simple object $Y$ is self-dual $\tr(\theta_{I(Y)}^2) = \pm \dim(\mathcal{C})$. We apply this result to $Y$:
\vspace{-12pt}
\begin{singlespace}
\begin{align*}
\sum{\gamma_i^2\theta_i^2} &= \frac{-k}{2}\sqrt{9k^2 + 12} - 2rp \mp \frac{\sqrt{3}}{2}i(k^2 - 4rp) \pm \frac{\dim(\mathcal{C})}{d} \\
&= \frac{-k}{2}\sqrt{9k^2 + 12} - 2rp \mp \frac{\sqrt{3}}{2}i(k^2 - 4rp) \pm \sqrt{9k^2 + 12}.
\end{align*}
\end{singlespace}

We now have expressions for a sum of roots of unity and the sum of their squares equal to some quadratic irrationalities in the same field. We will use the following result, which is proved in section \ref{rootssection}, in the proof of Theorem \ref{fam1}. In the statement, $\phi$ is the Euler-phi function.

\begin{prop} \label{rootsprop}
Let $c = \prod_{i = 1}^t{p_i}$ be an odd integer divisible by $3$ where the $p_i$ are distinct primes, and let $a, b, d \in \zz_{\geq 0}$. If there are $N$ roots of unity $\theta_i$ such that $\sum{\theta_i} = -a - b\sqrt{c}$ and $\sum{\theta_i^2} = -a - d\sqrt{c}$, then $N \geq P(a, b, c, d)$, where \[P(a, b, c, d) = \begin{cases} b\phi(c) + d\phi(c) + b + 2a &\text{if $c \equiv 3 \mod 4$ and $t$ is odd} \\ b\phi(c) - 2b + 2a &\text{otherwise.}\end{cases} \]
\end{prop}

We will also need the following lemma.
\begin{lm} \label{philm}
Let $c = 3x \geq 33$ be an odd squarefree integer. Then $\frac{\phi(c)}{\sqrt{c}} \geq \frac{\phi(33)}{\sqrt{33}}.$ Similarly, if $c = 3x \geq 21$, then $\frac{\phi(c)}{\sqrt{c}} \geq \frac{\phi(21)}{\sqrt{21}}$.
\end{lm}
\begin{proof}
Since $\frac{\phi(c)}{\sqrt{c}} = \frac{\phi(3)}{\sqrt{3}} \cdot \frac{\phi(x)}{\sqrt{x}}$, it is sufficient to show that $\frac{\phi(x)}{\sqrt{x}} \geq \frac{\phi(11)}{\sqrt{11}}$. If $x$ is prime, then 
$\frac{\phi(x)}{\sqrt{x}} = \frac{x - 1}{\sqrt{x}} \geq \frac{11 - 1}{\sqrt{11}} = \frac{\phi(11)}{\sqrt{11}}$.
Otherwise, $x$ is divisible by at least two primes. Since $2$ and $3$ do not divide $x$, we have
$\frac{\phi(x)}{\sqrt{x}} \geq \frac{\phi(5)\phi(7)}{\sqrt{5 \cdot 7}} \geq \frac{\phi(11)}{\sqrt{11}}$. Replacing $11$ with $7$, the argument with $21$ becomes identical.
\end{proof}
\begin{thm} \label{fam1}
Assume the based ring $K_1(e)$ is categorifiable. Then $e = 0, 2, 3, 6$.
\end{thm}
\begin{proof}
Consider the following sum of roots of unity and the sum of their squares:
\vspace{-12pt}
\begin{singlespace}
\begin{align}
\sum{\gamma_i^2\theta_i} + \sum{\gamma_i^2\bar{\theta_i}} &= -k\sqrt{9k^2 + 12} - 4rp \label{cc1}\\
\sum{\gamma_i^2\theta_i^2} + \sum{\gamma_i^2\bar{\theta_i^2}} &= (-k \pm 2)\sqrt{9k^2 + 12} - 4rp \label{cc2}
\end{align}
\end{singlespace}
\noindent

By Proposition \ref{Yprop}, \eqref{cc1} and \eqref{cc2} contain at most $2 \cdot \sum{\gamma_i^2} = 12 + 6k^2 + 8rp$ roots of unity. Let $c$ be the squarefree part of $9k^2 + 12$, and let $a = 4rp$, $b = k\sqrt{\frac{9k^2 + 12}{c}}$, and $d = (k \pm 2)\sqrt{\frac{9k^2 + 12}{c}}$.

\noindent
If {\boldmath $c = 3$}, then Proposition \ref{rootsprop} gives
\vspace{-12pt}
\begin{singlespace}
\begin{align*}
12 + 6k^2 + 8rp &\geq P(a, b, c, d) \geq k\sqrt{3k^2 + 4} \cdot 2 + (k - 2)\sqrt{3k^2 + 4} \cdot 2 + k\sqrt{3k^2 + 4} + 8rp \\
\Rightarrow \quad 12 + 6k^2 &\geq (5k - 4)\sqrt{3k^2 + 4}.
\end{align*}
\end{singlespace}
\noindent
This fails for $k > 3$. When $k = 3$, we have $c = 93$, so $k \leq 2$.

\medskip
\noindent
When {\boldmath{$c \neq 3$}}, Proposition \ref{rootsprop} gives
\vspace{-12pt}
\begin{singlespace}
\begin{align*}
12 + 6k^2 + 8rp &\geq P(a, b, c, d) \geq b\phi(c) - 2b + 2a \\
&\geq k\sqrt{\frac{9k^2 + 12}{c}}\phi(c) - 2k\sqrt{\frac{9k^2 + 12}{c}} + 8rp \\
12 + 6k^2 &\geq \frac{\phi(c)}{\sqrt{c}}k\sqrt{9k^2 + 12} - 2k\sqrt{\frac{9k^2 + 12}{c}}.
\intertext{If {\boldmath{$c = 21$}}, the above equation implies $k \leq 3$; when $k = 3$ we have $k = 93$, so $k \leq 2$. When {\boldmath{$c \geq 33$}}, using Lemma \ref{philm},}
12 + 6k^2 &\geq \frac{\phi(33)}{\sqrt{33}}k\sqrt{9k^2 + 12} - 2k\sqrt{\frac{9k^2 + 12}{33}} \quad \Rightarrow \quad k \leq 2.
\end{align*}
\end{singlespace}
\noindent


Thus we have shown that $k \leq 2$. This implies $e = 0, 3, 6$. However, recall that we have assumed $e \neq 2$. Hence, if $K_1(e)$ admits categorification then $e = 0, 2, 3, 6$.
\end{proof}

\section{Categorifications of $K_2(c)$} \label{fam2sec}
We follow a similar technique as in the previous section. Throughout this section, we assume that $\mathcal{C}$ is a fusion category with $K(\mathcal{C}) \simeq K_2(c)$ and arrive at a contradiction when $c > 2$. $K_2(0)$ is the based ring (5) of Theorem 1.1, and when $c = 1, 2$ we get (6).
Recall that the multiplication in $K_2(c)$ is given by
\vspace{-12pt}
\begin{singlespace}
\begin{align*}
X^2 &= cX + Y + cZ & Y^2 &= \mathbf{1} & Z^2 &= cX + Y + cZ \\
XY &= YX = Z & YZ &= ZY = X & XZ &= ZX = \mathbf{1} + cX + cZ
\end{align*}
\end{singlespace}

\subsection{Induction and Forgetful Functors} As before, we describe the simple objects in $\mathcal{Z}(\mathcal{C})$. Propositions \ref{so1} and \ref{so2} give the dimensions of simple objects and their images under the forgetful functor $F$.
Let $d = c + \sqrt{c^2 + 1}$.
Note that $d$ is irrational for $c > 0$. Since $K_2(0)$ is the Grothendieck ring of the category of representations of $\mathbb{Z}/4\mathbb{Z}$, and hence already known to admit categorification, we will assume that $d$ is irrational.
\begin{prop}
The dimensions of simple objects in $\mathcal{C}$ are $\dim(X) = \dim(Z) = d$ and $\dim(Y) = \dim(\mathbf{1}) = 1$, whence $\dim(\mathcal{C}) = 4 + 4cd$.
\end{prop}
\begin{proof}
The matrix for left multiplication by $X$ is
\vspace{-12pt}
\begin{singlespace}
\[M_X = \left(\begin{array}{cccc}
0 & 0 & 0 & 1 \\
1 & c & 0 & c \\
0 & 1 & 0 & 0 \\
0 & c & 1 & c
\end{array}\right).\]
\end{singlespace}
\noindent
The characteristic polynomial of $M_X$ is $x^4 - 2cx^3 - 2cx - 1 = (x^2 + 1)(x^2 - 2cx - 1)$, 
so $\dim(X) = \frac{2c + \sqrt{4c^2 + 4}}{2} = c + \sqrt{c^2 + 1}$. 
We know $\dim(X) = \dim(Z)$ because $X^* = Z$. The matrix for left multiplication by $Y$ is
\vspace{-12pt}
\begin{singlespace}
\[M_Y = \left(\begin{array}{cccc}
0 & 0 & 1 & 0 \\
0 & 0 & 0 & 1 \\
1 & 0 & 0 & 0 \\
0 & 1 & 0 & 0
\end{array}\right),\]
\end{singlespace}
\noindent
which has characteristic polynomial $(x + 1)^2(x - 1)^2$, so $\dim(Y) = \dim(\mathbf{1}) = 1$. We calculate
$\dim(\mathcal{C}) = \dim(\mathbf{1})^2 + \dim(X)^2 + \dim(Y)^2 + \dim(Z)^2 = 2 + 2d^2 = 4 + 4cd.$
\end{proof}
\begin{prop}  \label{so1}
There exist non-isomorphic simple objects $A, B, C, D, E, G, H$ in $\mathcal{Z}(\mathcal{C})$ for which $I(\mathbf{1}) = \mathbf{1} \op A \op B \op C$ and $I(Y) = D \op E \op G \op H$ and
\vspace{-12pt}
\begin{singlespace}
\begin{align*}
F(A) &= \mathbf{1} \op cX \op cZ & \dim(A) &= 1 + 2cd \\
F(B) &= \mathbf{1} \op gX \op hZ & \dim(B) &= 1 + cd \\
F(C) &= \mathbf{1} \op hX \op gZ & \dim(C) &= 1 + cd \\
F(D) &= jX \op Y \op kZ & \dim(D) &= 1 + (j+k)d \\
F(E) &= lX \op Y \op mZ & \dim(E) &= 1 + (l+m)d \\
F(G) &= nX \op Y \op pZ & \dim(G) &= 1 + (n+p)d \\
F(H) &= qX \op Y \op rZ & \dim(H) &= 1 + (q+r)d
\end{align*}
\end{singlespace}
\noindent
where $g + h = c$ and $j + l + n + q = k + m + p + r = 2c$
\end{prop}
\begin{proof}
Two of the formal codegrees of $K_1(c)$ are $f_1 = f_2 = 4$. (see Remark \ref{fc=4}). By Theorem 2.13 of \cite{ostrik}, $I(\mathbf{1}) \in \mathcal{Z}(\mathcal{C})$ decomposes into the sum of $4$ simple objects $I(\mathbf{1}) = 1 \op A \op B \op C$; we can assume that $\dim(B) = \dim(C) = \frac{\dim(\mathcal{C})}{4} = 1 + cd$ whence $\dim(A) = \frac{\dim(\mathcal{C})}{2} - 1 = 1 + 2cd$. By Proposition 5.4 of \cite{ENO},
\vspace{-12pt}
\begin{singlespace}
\begin{align*}
F(I(\mathbf{1})) &= \bigoplus_{X \in \mathbb{O}(\mathcal{C})} X \ot \mathbf{1} \ot X^* \\
&= \mathbf{1} \op (\mathbf{1} \op cX \op cZ) \op \mathbf{1} \op (\mathbf{1} \op cX \op cZ) \\
&= 4 \cdot \mathbf{1} \op 2cX \op 2cZ
\end{align*}
\end{singlespace}
\noindent
and $1 + \dim(A) + \dim(B) + \dim(C) = 4 + 4cd$. 
Let $F(A) = \mathbf{1} \op aX \op bZ$, whence $1 + d(a+b) = \dim(A) = 1 + 2cd$ implies $a + b = 2c$.
Since this decomosition into simple objects is unique and $I(\mathbf{1})$ is self-dual and the dimension of $A$ is unique, $A$ must be self-dual. Hence, $a = b = c$.

Now let $F(B) = \mathbf{1} \op gX \op hZ$, whence $1 + (g+h)d = \dim(B) = 1 + cd$ implies $g + h = c$. By duality, $F(C) = \mathbf{1} \op hX \op gZ$. (In the case that $B$ and $C$ are self-dual, we have $g = h$ so this is true in general).

Since $\dimhom(I(\mathbf{1}), I(Y)) = \dimhom(F(I(\mathbf{1}), Y) = 0$, the simple objects in the decomposition of $I(Y)$ are distinct from $A, B, C$. By Proposition 5.4 of \cite{ENO},
\vspace{-12pt}
\begin{singlespace}
\begin{align*}
F(I(Y)) &= \bigoplus_{X \in \mathbb{O}(\mathcal{C})} X \ot Y \ot X^* \\
&= Y \op (cX \op Y \op cZ) \op Y \op (cX \op Y \op cZ) \\
&= 2cX \op 4Y \op 2cZ
\end{align*}
\end{singlespace}
\noindent
Hence, $\dimhom(I(Y), I(Y)) = \dimhom(F(I(Y)), Y) = 4$, so either $I(Y) = 2D$ for some simple object $D$ or $I(Y) = D \op E \op G \op H$ for distinct simple objects $D, E, G, H$. First assume that $I(Y) = 2D$. Then, $F(D) = cX \op 2Y \op cZ$, whence $\dim(D) = 2 + 2cd$. Let $\theta_D$ be the balance isomorphism. Now we have 
\vspace{-12pt}
\begin{singlespace}
\[0 = \tr\theta_{I(Y)} = 2\dim(D)\theta_D = 2(2 + 2cd)\theta_D\]
\end{singlespace}
\noindent
which implies $\theta_D = 0$, a contradiction. Hence $I(Y) = D \op E \op G \op H$.
\end{proof}

\begin{prop} \label{so2}
There exist simple objects $L_i$ distinct from $A, B, C, D, E, G, H$ in $\mathcal{Z}(\mathcal{C})$ such that 
\begin{enumerate}
\item $F(L_i) = \gamma_iX + \gamma_i^*Z$
\item $I(X) = cA \op gB \op hC \op jD \op lE \op nG \op qH \op \sum{\gamma_iL_i}$
\item $\sum{\gamma_i(\gamma_i + \gamma_i^*)} \leq 4 + 3c^2$
\end{enumerate}
\end{prop}
\begin{proof}
Let $L_i = \phi_i \cdot \mathbf{1} \op \gamma_iX \op \psi_iY \op \gamma_i^*Z$ such that (1) is satisfied. (We can assume $\gamma_i \neq 0$ because $\dimhom(L_i, I(X)) = \dimhom(F(L_i), X) = \gamma_i$. That is, if $\gamma_i = 0$ then $L_i$ does not appear in the decomposition of $I(X)$.)
Then we compute
\vspace{-12pt}
\begin{singlespace}
\begin{align*}
F(I(X)) &= cF(A) \op gF(B) \op hF(C) \op jF(D) \op lF(E) \op nF(G) \op qF(H) \op \sum{\gamma_iF(L_i)} \\
&= c(\mathbf{1} \op cX \op cZ) \op g(\mathbf{1} \op gX \op hZ) \op h(\mathbf{1} \op hX \op gZ) \op j(jX \op Y \op kZ) \\
& \quad \op l(lX \op Y \op mZ) \op n(nX \op Y \op pZ) \op q(qX \op Y \op rZ) \\
& \quad \op \sum{\gamma_i(\phi_i \cdot \mathbf{1} \op \gamma_iX \op \psi_iY \op \gamma_i^*Z)} \\
&= (c + g + h + \sum{\gamma_i\phi_i})\cdot\mathbf{1} \op (c^2 + g^2 + h^2 + j^2 + l^2 + n^2  + q^2 + \sum{\gamma_i^2})X \\
& \quad \op (j + l + n + q + \sum{\gamma_i\psi_i})Y \op (c^2 + 2gh + jk + lm + np + qr + \sum{\gamma_i\gamma_i^*})Z
\end{align*}
\end{singlespace}
\noindent
Using Theorem 5.4 of \cite{ENO}, we compute
\vspace{-12pt}
\begin{singlespace}
\[F(I(X)) = \bigoplus_{Y \in \mathbb{O}(\mathcal{C})} Y \ot X \ot Y^* = 2c \cdot \mathbf{1} + (4 + 4c^2)X + 2cY + 4c^2Z.\]
\end{singlespace}
\noindent
Hence, $c + g + h + \sum{\gamma_i\phi_i} = 2c$ implies $\phi_i = 0$ for all $i$. Also $j + l + n + q + \sum{\gamma_i\psi_i} = 2c$ implies $\psi_i = 0$ for all $i$. 
 
We are left with the following two equations:
\vspace{-12pt}
\begin{singlespace}
\begin{align}
c^2 + g^2 + h^2 + (j^2 + l^2 + n^2  + q^2) + \sum{\gamma_i^2} &= 4 + 4c^2 \label{squares} \\
c^2 + 2gh + jk + lm + np + qr + \sum{\gamma_i\gamma_i^*} &= 4c^2 \label{dubproducts} \\
\intertext{Because of duality constraints, the set $\{j, l, n, q\}$ is a permutation of the set $\{k, m, p, r\}$ so, \eqref{squares} implies}
c^2 + g^2 + h^2 + (k^2 + m^2 + p^2  + r^2) + \sum{\gamma_i^2} &= 4 + 4c^2. \label{squarefriends}
\end{align}
\end{singlespace}
\noindent
We add \eqref{squares} with \eqref{squarefriends} and twice \eqref{dubproducts} to get
\vspace{-12pt}
\begin{singlespace}
\[4c^2 + 2(g + h)^2 + (j + k)^2 + (l + m)^2 + (n + p)^2 + (q + r)^2 + 2 \cdot \sum{\gamma_i^2} + 2\cdot\sum{\gamma_i\gamma_i^*} = 8 + 16c^2 \]
\end{singlespace}
\noindent
whence,
\vspace{-12pt}
\begin{singlespace}
\begin{align*}
2 \cdot \sum{\gamma_i(\gamma_i + \gamma_i^*)} &= 8 + 10c^2 - \left((j + k)^2 + (l + m)^2 + (n + p)^2 + (q + r)^2\right) \\
&\leq 8 + 10c^2 - \frac{1}{4}(j+k+l+m+n+p+q+r)^2 \\
&= 8 + 10c^2 - 4c^2 \\
\Rightarrow  \sum{\gamma_i(\gamma_i + \gamma_i^*)} &\leq 4 + 3c^2 &&\qedhere
\end{align*}
\end{singlespace}
\noindent
\end{proof}

\vspace{-.3in}
\subsection{Twists} Using the same methods as in section 4, we calculate the twists for some of the simple objects $\mathcal{Z}(\mathcal{C})$.

\begin{prop} The twists $\theta_A = \theta_B = \theta_C = 1$ \end{prop}
\begin{proof}
By Theorem 2.5 of \cite{ostrik}, 
\vspace{-12pt}
\begin{singlespace}
\begin{align*}
4 + 4cd &= \dim(\mathcal{C}) = \tr(\theta_{I(\mathbf{1})}) \\
&= 1 + \dim(A)\theta_A + \dim(B)\theta_B + \dim(C)\theta_C \\
&= 1 + (1 + 2cd)\theta_A + (1 + cd)\theta_B + (1 + cd)\theta_C && \qedhere
\end{align*}
\end{singlespace}
\noindent
\end{proof}
Let $\theta = \theta_D$.
\begin{prop} \label{thetacases}
Without loss of generality, the balance isomorphisms $\theta_D, \theta_E, \theta_G, \theta_H$ satisfy one of the following
\begin{enumerate}
\item $\theta = \theta_D = \theta_E = -\theta_G = -\theta_H = 1$
\item $\theta = \theta_D = \theta_E = -\theta_G = -\theta_H = i$
\end{enumerate}
\end{prop}
\begin{proof}
By Theorem 2.7 of \cite{ostrik}, 
\vspace{-12pt}
\begin{singlespace}
\begin{align*}
\pm (4 + 4cd) &= \pm \dim(\mathcal{C}) = \tr(\theta_{I(Y)}^2) \\
&= \dim(D)\theta_D^2 + \dim(E)\theta_E^2 + \dim(G)\theta_G^2 + \dim(H)\theta_H^2 \\
&= (1 + (j+k)d)\theta_D^2 + (1 + (l+m)d)\theta_E^2 + (1 + (n+p)d)\theta_G^2 + (1 + (q+r)d)\theta_H^2
\end{align*}
\end{singlespace}
\noindent
Then, since $j+k+l+m+n+p+q+r = 4c$ we have $\theta_D^2 = \theta_E^2 = \theta_G^2 = \theta_H^2 = \pm 1$.

Next, Theorem 2.5 of \cite{ostrik} gives 
\vspace{-12pt}
\begin{equation} \label{tY}
0 = \tr(\theta_{I(Y)}) = (1 + (j+k)d)\theta_D + (1 + (l+m)d)\theta_E + (1 + (n+p)d)\theta_G + (1 + (q+r)d)\theta_H.
\end{equation}
\noindent
 Since the twists are fourth roots of unity and $d$ is real and irrational, we have
$\theta_D + \theta_E + \theta_G + \theta_H = 0$, which implies, without loss of generality, $\theta_D = \theta_E = -\theta_G = -\theta_H$.
\end{proof}

\subsection{Categorifications of $K_2(c)$} Let $\theta_i = \theta_{L_i}$. Using the results in \cite{NgS} and calculating the traces of $\theta_{I(X)}$ and $\theta_{I(X)}^2$, we arrive at expressions for linear combinations of the $\theta_i$ and $\theta_i^2$.  Using Theorem 2.5 of \cite{ostrik}, we compute
\vspace{-12pt}
\begin{singlespace}
\begin{align*}
0 &= \tr(\theta_{I(X)}) \\
&= c\dim(A)\theta_A + g\dim(B)\theta_B + h\dim(C)\theta_C + j\dim(D)\theta_D + l\dim(E)\theta_E \\
&\quad + n\dim(G)\theta_G + q\dim(H)\theta_H + \sum{\gamma_i\dim(L_i)\theta_i} \\
&= c(1+2cd) + g(1+cd) + h(1+cd) + j(1 + (j+k))d)\theta + l(1 + (l+m)d)\theta \\
&\quad - n(1 + (n+p)d)\theta - q(1 + (q+r)d)\theta + \sum{\gamma_i(\gamma_i + \gamma_i^*)d\theta_i} \\
&= 2c + 3c^2d + \theta(j + l - n - q) + \theta d(j^2 + jk + l^2 + lm - n^2 - np - q^2 - qr) \\
&\quad + \sum{\gamma_i(\gamma_i + \gamma_i^*)d\theta_i},
\end{align*}
\end{singlespace}
\noindent
whence,
\begin{equation}
\sum{\gamma_i(\gamma_i + \gamma_i^*)\theta_i} = -\frac{2c}{d} - 3c^2 - \frac{\theta}{d}(j + l - n - q) - \theta (j^2 + jk + l^2 + lm - n^2 - np - q^2 - qr). \label{sumofthetas}
\end{equation}
Also by \cite{NgS},
\vspace{-12pt}
\begin{singlespace}
\begin{align*}
0 &= \tr(\theta_{I(X)}^2) \\
&= c\dim(A)\theta_A^2 + g\dim(B)\theta_B^2 + h\dim(C)\theta_C^2 + j\dim(D)\theta_D^2 + l\dim(E)\theta_E^2 + n\dim(G)\theta_G^2 \\
&\quad + q\dim(H)\theta_H^2 + \sum{\gamma_i\dim(L_i)\theta_i^2} \\
&= 2c + 3c^2d + \theta^2(j + l + n + q) + \theta^2 d(j^2 + jk + l^2 + lm + n^2 + np + q^2 + qr) \\
&\quad + \sum{\gamma_i(\gamma_i + \gamma_i^*)d\theta_i^2} \\
&= 2c + 3c^2d + 2c\theta^2 + \theta^2 d(j^2 + jk + l^2 + lm + n^2 + np + q^2 + qr) + \sum{\gamma_i(\gamma_i + \gamma_i^*)d\theta_i^2},
\end{align*}
\end{singlespace}
\noindent
whence
\begin{equation}
\sum{\gamma_i(\gamma_i + \gamma_i^*)\theta_i^2} = -\frac{2c}{d} - 3c^2 - \frac{2c}{d}\theta^2 - \theta^2 (j^2 + jk + l^2 + lm + n^2 + np + q^2 + qr). \label{sumofthetas^2}
\end{equation}

Equations \ref{sumofthetas} and \ref{sumofthetas^2} show that a sum of a limited number of roots of unity and the sum of their squares are equal to some quadratic irrationalities. We will be able to get a contradiction when $c > 2$ using the following results from section \ref{rootssection}:

\begin{prop} \label{sqrt2}
For integers $a, b$, it takes at least $|a| + 2|b|$ roots of unity to write $a + b\sqrt{2}$ as a sum of roots of unity.
\end{prop}

\begin{thm} \label{a+bsqrtc}
For $d$ square-free it requires at least $|b|\phi(2d)$ roots of unity to write $a + b\sqrt{d}$ as a sum of roots of unity.
\end{thm}

\begin{thm} \label{fam2}
If the based ring $K_2(c)$ is categorifiable, then $c \leq 2$.
\end{thm}
\begin{proof}
Assume there exists a category $\mathcal{C}$ such that $K(\mathcal{C}) = K_2(c)$. We now consider the two cases outlined in Proposition \ref{thetacases}.
\medskip
\paragraph{\textbf{Case 1:} {\boldmath{$\theta = 1$}}}
Then, \eqref{sumofthetas^2} becomes
\vspace{-12pt}
\begin{singlespace}
\begin{align*}
\sum{\gamma_i(\gamma_i + \gamma_i^*)\theta_i^2} &= -\frac{4c}{d} - 3c^2 - (j^2 + jk + l^2 + lm + n^2 + np + q^2 + qr) \\
&= -4c\sqrt{c^2 + 1} + 4c^2 - 3c^2 - (j^2 + jk + l^2 + lm + n^2 + np + q^2 + qr)
\end{align*}
\end{singlespace}
\noindent
The right hand side of the above equation requires at least
\vspace{-12pt}
\begin{singlespace}
\[4c\sqrt{\frac{c^2 + 1}{2}}\cdot 2 \geq \frac{8}{\sqrt{2}}c^2\]
\end{singlespace}
\noindent
roots of unity to write, while the left hand side contains at most $\sum{\gamma_i(\gamma_i + \gamma_i^*)}$ roots of unity. Hence,
\vspace{-12pt}
\begin{singlespace}
\[4 + 3c^2 \geq \sum{\gamma_i(\gamma_i + \gamma_i^*)} \geq \frac{8}{\sqrt{2}}c^2,\]
\end{singlespace}
\noindent
which implies $c \leq 1$. Note that if $c = 1$, then we have $4c\sqrt{\frac{c^2+1}{2}} \cdot 2 = 8 > 7 = 4 + 3c^2$, so in fact, $c < 1$.
\medskip
\paragraph{\textbf{Case 2:} {\boldmath{$\theta = i$}}}
Then, \eqref{sumofthetas} becomes
\vspace{-12pt}
\begin{singlespace}
\begin{align*}
&\sum{\gamma_i(\gamma_i + \gamma_i^*)\theta_i} = -\frac{2c}{d} - 3c^2 - \frac{i}{d}(j + l - n - q) - i(j^2 + jk + l^2 + lm - n^2 - np - q^2 - qr)\\
&= -2c\sqrt{c^2 + 1} - c^2 - \frac{i}{d}(j + l - n - q) - i(j^2 + jk + l^2 + lm - n^2 - np - q^2 - qr),
\end{align*}
\end{singlespace}
\noindent
whence
\begin{equation}
\sum{\gamma_i(\gamma_i + \gamma_i^*)\theta_i} + \sum{\gamma_i(\gamma_i + \gamma_i^*)\bar{\theta_i}} = -4c\sqrt{c^2 + 1} -2c^2 \label{eq1}
\end{equation}

\medskip
\noindent
\textit{If $c^2 + 1$ is $2$ times a square}, then Proposition \ref{sqrt2} implies that right hand side of \eqref{eq1} requires at least
\vspace{-12pt}
\begin{singlespace}
\[4c\sqrt{\frac{c^2 + 1}{2}} \cdot 2 + 2c^2 \geq \left(\frac{8}{\sqrt{2}} + 2\right)c^2\]
\end{singlespace}
\noindent
roots of unity to write, while the left hand side contains at most $2 \cdot \sum{\gamma_i(\gamma_i + \gamma_i^*)}$ roots of unity. Hence, 
\vspace{-12pt}
\begin{singlespace}
\[8 + 6c^2 \geq 2 \cdot \sum{\gamma_i(\gamma_i + \gamma_i^*)} \geq \left(\frac{8}{\sqrt{2}} + 2\right)c^2,\]
\end{singlespace}
\noindent
which implies $c \leq 2$.

\medskip
\noindent
\textit{If $c^2 + 1$ is not $2$ times a square}, then Theorem \ref{a+bsqrtc} implies that the right hand side of \eqref{eq1} requires at least
\vspace{-12pt}
\begin{singlespace}
\[4c\sqrt{\frac{c^2 + 1}{5}} \cdot 4 \geq \frac{16}{\sqrt{5}}c^2\]
\end{singlespace}
\noindent
roots of unity to write, so
\vspace{-12pt}
\begin{singlespace}
\[8 + 6c^2 \geq 2 \cdot \sum{\gamma_i(\gamma_i + \gamma_i^*)} \geq \frac{16}{\sqrt{5}}c^2,\]
\end{singlespace}
\noindent
which implies $c \leq 2$.
\end{proof}

\section{Sums of Roots of Unity} \label{rootssection}
In the proofs of Theorems \ref{fam1} and \ref{fam2}, we needed tight bounds on the number of roots of unity required to write certain quadratic irrationalities. To develop these bounds, we will make use the following observation: suppose we have $N$ roots of unity $\theta_i$ in $\qq(\zeta_x)$ 
such that $\sum{\theta_i} = a + b\sqrt{c}$ for integers $a, b, c$. Then summing over all Galois conjugates
(by $\gal(\qq(\zeta_x)/\qq(\sqrt{c}))$)
of this equation gives a collection of $MN$ roots of unity, where $M = |\gal(\qq(\zeta_x)/\qq(\sqrt{c}))|$,
whose sum is $M(a + b\sqrt{c})$ and whose multiplicities are Galois-invariant. Hence, if it requires at least $MA$
roots of unity to write $M(a + b\sqrt{c})$ such that the multiplicities are Galois-invariant, it requires at least $A$ roots of unity to write $a + b\sqrt{c}$. This reduces the problem to a question about sums of Galois orbits.

\begin{lm} \label{rtslm}
Let $a, b, c$ be integers with $c = 2^\alpha\prod{p_i}$ the product of distinct primes. 
Define $\epsilon = 0$ if $c = 1$ (mod $4$) and $\epsilon = 1$ if $c = 2, 3$ (mod $4$). 
Assume there exist $N$ roots of unity $\theta_i$ in some cyclotomic field $\qq(\zeta_x)$ such that
\begin{equation}
\sum{\theta_i} = a + b\sqrt{c} \label{bee}
\end{equation}
Then, for some positive integer $M$, there exists a collection of less than or equal to $MN$ roots of unity $\theta_i'$ with 
\begin{equation}
\sum{\theta_i'} = M(a + b\sqrt{c})
\end{equation}
such that the multiplicity of $\theta_i'$ is $\gal(\qq(\zeta_x)/\qq(\sqrt{c}))$-invariant, and the $\theta_i'$ are all $(2^{\epsilon + 1}c)^{\text{th}}$ roots of unity.
\end{lm}
\begin{proof}
Define $n = 2^{2\epsilon}c$. Let $M = |\gal(\qq(\zeta_x)/\qq(\sqrt{c}))|$. Summing all Galois conjugates of \eqref{bee} by $\gal(\qq(\zeta_x)/\qq(\sqrt{c}))$ gives an expression in which the multiplicity of $\theta_i'$ is Galois-invariant.
Consider some Galois orbit $O$ and suppose that all elements of $O$ are primitive $Y^{\text{th}}$ roots of unity where $Y = 2^\beta \prod{p_i^{r_i}} \prod{q_i^{s_i}}$.

{\textbf{\boldmath{Case 1: $n \nmid Y$}}}
Then $\sqrt{c} \notin \qq(\zeta_Y)$, so the orbit contains all primitive $Y^{\text{th}}$ roots of unity and one of the following three things happens:
\begin{enumerate}
\item $\sum_{\zeta \in O}{\zeta} = 1$; we replace $O$ by $1$.
\item $\sum_{\zeta \in O}{\zeta} = -1$; we replace $O$ by $-1$.
\item $\sum_{\zeta \in O}{\zeta} = 0$; we drop this orbit from our sums.
\end{enumerate}
All of these replacements have size less than or equal to $|O|$ and all contain $(2^{\epsilon + 1}c)^{\text{th}}$ roots of unity.

{\textbf{\boldmath{Case 2: $n \mid Y$ and $\beta > 1 + \epsilon$ or $r_l > 1$ for some $l$.}}} Define $k = 2$ if $\beta > 1 + \epsilon$, and $k = p_l$ otherwise. Note that, $p_j \mid Y/k$ for all $j$, and $2 \mid Y/k$, so if $i$ is relatively prime to $Y$ then $Y/k + i$ is relatively prime to $Y$. Hence, for any primitive $Y^{\text{th}}$ root of unity $\zeta_Y^i$,
\vspace{-12pt}
\begin{singlespace}
\[\zeta_k \cdot \zeta_Y^i = \zeta_Y^{Y/k + i}\]
\end{singlespace}
\noindent
is also a primitive $Y^{\text{th}}$ root of unity. In addition, $n \mid Y/k$ implies that $Y/k + i  = i$ mod $n$, so $\frac{Y/k + i}{i} = 1$ mod $n$. That is,
\vspace{-12pt}
\begin{singlespace}
\[\zeta_k \cdot \zeta_Y^i = \zeta_Y^{Y/k + i} = \left(\zeta_Y^{i}\right)^{\frac{Y/k + i}{i}}\]
\end{singlespace}
\noindent
is Galois conjugate to $\zeta_Y^i$ over $\qq(\zeta_n) \supseteq \qq(\sqrt{c})$. Hence,
\vspace{-8pt}
\begin{singlespace}
\[\sum_{\zeta \in O}{\zeta} = \sum_{\zeta \in O}{\zeta_k \cdot \zeta} = \zeta_k\sum_{\zeta \in O}{\zeta}.\]
\end{singlespace}
\noindent
Since $k > 1$ this implies $\sum_{\zeta \in O} \zeta = 0$. Thus we can drop $O$ from our sums.

{\textbf{\boldmath{Case 3: $n \mid Y$ and $\beta \leq 1 + \epsilon$ and $r_l \leq 1$ for all $l$.}}} Let $u = \prod{q_i^{s_i}}$ and $v = 2^\beta\prod{p_i^{r_i}}$ so that $Y = uv$ and $(u, v) = 1$. By the Chinese Remainder Theorem, any $Y^{\text{th}}$ root of unity can be written uniquely as a $u^{\text{th}}$ root of unity times a $v^{\text{th}}$ root of unity, and $\gal(\qq(\zeta_Y)/\qq) = \gal(\qq(\zeta_u)/\qq) \oplus \gal(\qq(\zeta_v)/\qq)$. 
Any element of $\gal(\qq(\zeta_u)/\qq)$ can be seen as an element of $\gal(\qq(\zeta_Y)/\qq(\sqrt{c}))$. (If $\sqrt{c} \in \qq(\zeta_Y)$, then $n \mid Y$; since $(n, u) = 1$, it follows that $n \mid v$ so $\sqrt{c} \in \qq(\zeta_v)$.) 
Thus, there is an orbit $O'$ of primitive $v^\text{th}$ roots of unity so that $O = \{\zeta_u \zeta_v : \zeta_u \ \text{is a primitive $u^\text{th}$ root of unity and} \ \zeta_v \in O'\}$. Hence,
\vspace{-12pt}
\begin{singlespace}
\[\sum_{\zeta \in O}{\zeta} = \sum_{\zeta_v \in O'}{\zeta_v} \ \cdot \sum_{i, (i, u) = 1}{\zeta_u}\]
\end{singlespace}
\noindent
If $u$ is not squarefree then $\sum_{\zeta \in O}{\zeta} = 0$ so we can remove $O$ from our sum. If $u$ is a squarefree product of an even number of primes, $\sum_{\zeta \in O}{\zeta} = \sum_{\zeta_v \in O'}{\zeta_v}$, so we can replace $O$ by $O'$ (which is an orbit containing $2^{\epsilon + 1}c^{\text{th}}$ roots of unity because $v \mid 2^{\epsilon + 1}c$). Lastly, if $u$ is a squarefree product of an odd number primes, $\sum_{\zeta \in O}{\zeta} = -1 \cdot \sum_{\zeta_v \in O'}{\zeta_v}$, so we replace $O$ by $-1 \cdot O'$. This is a collection of LCM$(2,v)^{\text{th}}$ roots of unity, and since LCM$(2,v) \mid 2^{\epsilon + 1}c$, this leaves us with a Galois-invariant collection of $2^{\epsilon + 1}c^{\text{th}}$ roots of unity.
\end{proof}

\begin{thm}[Theorem \ref{a+bsqrtc}]
For $c$ square-free it requires at least $|b|\phi(2c)$ roots of unity to write $a + b\sqrt{c}$ as a sum of roots of unity.
\end{thm}
\begin{proof}
Assume to the contrary that we could write $a + b\sqrt{c}$ with less than $|b|\phi(2c)$ roots of unity in some cyclotomic field $\qq(\zeta_x)$. Let $n = 2c$ if $c \equiv 1 \mod 4$ and $n = 4c$ if $c \equiv 2, 3 \mod 4$. Then by Lemma~\ref{rtslm}, for some $M$, we could write $\sum{\theta_i'} = M(a + b\sqrt{c})$
with less than $M|b|\phi(c)$ roots of unity, where the $\theta_i'$ are all $n^{\text{th}}$ roots of unity whose multiplicities are $\gal(\qq(\zeta_x)/\qq(\sqrt{c}))$-invariant. Consider some orbit of primitive $Y^{\text{th}}$ roots of unity.

{\textbf{\boldmath{Case 1: $\sqrt{c} \notin \qq(\zeta_Y)$}}} The orbit consists of all primitive $Y^{\text{th}}$ roots of unity so the orbit sums to $0$ or $\pm 1$, and the coefficient of $\sqrt{c}$ in the sum of the orbit is $0$.

{\textbf{\boldmath{Case 2: $\sqrt{c} \in \qq(\zeta_Y)$ and $c \equiv 2, 3$ mod $4$}}} By Lemma~\ref{rtslm}, the only case that remains is $Y = 4c$.
Here there are two Galois orbits $O_1$ and $O_2$ of $4c^{\text{th}}$ roots of unity. We compute
\vspace{-12pt}
\begin{singlespace}
\[\sum_{\zeta \in O_1}{\zeta} + \sum_{\zeta \in O_2}{\zeta} = 0,\]
\end{singlespace}
\noindent
and when we take the difference, it is the product of Gauss sums (see \cite{Sam} Section 5.5, Equation 6): 
\vspace{-12pt}
\begin{singlespace}
\[\sum_{\zeta \in O_1}{\zeta} - \sum_{\zeta \in O_2}{\zeta} = 2\sqrt{c}\]
\end{singlespace}
\noindent
Hence, 
\vspace{-12pt}
\begin{singlespace}
\[\sum_{\zeta \in O_1}{\zeta} = \sqrt{c} \quad \text{and} \quad \sum_{\zeta \in O_2}{\zeta} = -\sqrt{c},\]
\end{singlespace}
\noindent and 
\vspace{-12pt}
\begin{singlespace}
\[|O_1| = |O_2| = \frac{1}{2}\phi(4c) = \phi(2c).\]
\end{singlespace}
\noindent

{\textbf{\boldmath{Case 3: $\sqrt{c} \in \qq(\zeta_Y)$ and $c \equiv 1$ mod $4$}}}
By Lemma~\ref{rtslm}, it remains to consider when $Y = c, 2c$. We make a similar calculation as above, taking the top sign when $c \equiv 5$ mod $8$ and the lower sign when $c \equiv 1$ mod $8$, to produce the table:

\begin{center}
\begin{tabular}{c|c|c|c} 
orbit & size & sum of orbit \\ 
\hline
$Y = c$ & $\frac{1}{2}\phi(2c)$ & $\frac{1}{2}\left((-1)^t + \sqrt{c}\right)$ \\ 
\hline
$Y = c$ & $\frac{1}{2}\phi(2c)$ & $\frac{1}{2}\left((-1)^t - \sqrt{c}\right)$ \\
\hline
$Y = 2c$ & $\frac{1}{2}\phi(2c)$ & $\frac{1}{2}\left((-1)^{t + 1} \pm \sqrt{c}\right)$ \\
\hline
$Y = 2c$ & $\frac{1}{2}\phi(2c)$ & $\frac{1}{2}\left((-1)^{t + 1} \mp \sqrt{c}\right)$ \\
\end{tabular}
\end{center}

In all cases, the coefficient of $\sqrt{c}$ over the size of the orbit has absolute value less than or equal to $1/\phi(2c)$. Hence it requires at leats $M|b|\phi(2c)$ roots of unity to write $M(a + b\sqrt{c})$ as the sum of roots of unity whose multiplicities are $\gal(\qq(\zeta_x)/\qq(\sqrt{c}))$-invariant.
\end{proof}

\begin{rem} In some cases the above bound is sharp. For example, when $c \equiv 1$ mod $4$ and $|a| \leq |b|$.
\end{rem}

In the case when $c = 2$ we need an even better bound on the roots of unity needed to represent $a + b\sqrt{c}$.

\begin{prop}[Proposition \ref{sqrt2}]
For integers $a, b$, it takes at least $|a| + 2|b|$ roots of unity to write $a + b\sqrt{2}$ as a sum of roots of unity.
\end{prop}
\begin{proof}
Assume to the contrary that we could write $a + b\sqrt{2}$ with less than $|a| + 2|b|$ roots of unity. 
Then, Lemma~\ref{rtslm} says that there exists some $M$ such that we can write $M(a + b\sqrt{2})$ as the sum of less than $M(a + 2b)$ roots of unity which are $8^{\text{th}}$ roots of unity whose multiplicities are invariant under the action of Gal$\left(\qq(\zeta_{8})/\qq(\sqrt{2})\right)$. For every $Y$ with $Y \mid 8$ there is either one orbit which contains all primitive $Y^{\text{th}}$ roots of unity or two orbits ($*$ and $\dagger$) which each contain half of the primitive $Y^{\text{th}}$ roots of unity. We write down the following table:
\begin{center}
\begin{tabular}{c|c|c} 
orbit & size & sum of orbit \\ 
\hline
$Y = 1 $ & $1$ & $1$\\ 
\hline
$Y = 2$ & $1$ & $-1$\\
\hline
$Y = 4$ & $2$ & $0$ \\
\hline
$Y = 8^*$ & $2$ & $\sqrt{2}$ \\
\hline
$Y = 8^{\dagger}$ & $2$ & $-\sqrt{2}$
\end{tabular}
\end{center}
Define the linear function $f(x + y\sqrt{2}) = x + 2y$. Note that for all orbits, we have
\vspace{-12pt}
\begin{singlespace}
\[-1 \leq \frac{1}{|O_i|}f\left(\sum_{\zeta \in O_i}{\zeta}\right)\]
\end{singlespace}
\noindent
Assume there were orbits $O_1, \ldots, O_n$ such that 
\vspace{-12pt}
\begin{singlespace}
\[\sum_i{\sum_{\zeta \in O_i}{\zeta}} = M(-a - b\sqrt{2}).\]
\end{singlespace}
\noindent
Then
\vspace{-12pt}
\begin{singlespace}
\begin{align*}
\sum_i{|O_i|} \cdot (-1) &\leq \sum_i{|O_i| \cdot \frac{1}{|O_i|}f\left(\sum_{\zeta \in O_i}{\zeta}\right)} = \sum_i{f\left(\sum_{\zeta \in O_i}{\zeta}\right)} = f\left(\sum_i{\sum_{\zeta \in O_i}{\zeta}}\right) \\
&= f\left(M(-a - b\sqrt{2})\right) = -Mf(a + b\sqrt{c}) = -M(a + 2b) \\
\Rightarrow \sum_i{|O_i|} &\geq M(a + 2b).
\end{align*}
\end{singlespace}
\noindent
Hence, the number of roots of unity needed to write $M(a + b\sqrt{2})$ is greater than $M(a + 2b)$. This is a contradiction, so it requires at least $|a| + 2|b|$ roots of unity to write $a + b\sqrt{2}$.
\end{proof}

When studying $K_1(e)$, we had a special case where the sum of the squares of the roots of unity is another given quadratic irrationality in the same field. The following lemma is analgous to lemma \ref{rtslm}, but keeps track of sums of squares of roots of unity as well.

\begin{lm} \label{newrtlm}
Let $c = 2^a\prod{p_i}$ be the product of distinct primes. Define $\epsilon = 0$ if $c = 1$ (mod $4$) and $\epsilon = 1$ if $c = 2, 3$ (mod $4$). 
Let $L = \text{LCM}(2^{\epsilon + 2}c, 3)$. Let $\alpha$ and $\beta$ be algebraic integers in $\zz(\sqrt{c})$. Assume there exist $N$ roots of unity $\theta_i$ in some cyclotomic field $\qq(\zeta_x)$ such that
\begin{equation}
\sum{\theta_i} = \alpha \quad \text{and} \quad
\sum{\theta_i^2} = \beta \label{theta-alpha-beta}.
\end{equation}
Then, for some positive integer $M$, there exists a collection of less than or equal to $MN$ roots of unity $\theta_i'$ with 
\begin{equation}
\sum{\theta_i'} = M\alpha \quad \text{and} \quad
\sum{(\theta_i')^2} = M\beta \label{a-b-prime},
\end{equation}
such that the multiplicity of $\theta_i'$ is $\gal(\qq(\zeta_x)/\qq(\sqrt{c}))$-invariant, and the $\theta_i'$ are all $L^{\text{th}}$ roots of unity.
\end{lm}
\begin{proof} Define $n = 2^{2\epsilon}c$. We argue exactly as in Lemma~\ref{rtslm}, but with the following modifications to keep track of the sums of the squares in each orbit:

{\textbf{\boldmath{Case 1: $n \nmid Y$}}}
One of the following four things happens:
\begin{enumerate}
\item $Y$ is odd, so $\sum_{\zeta \in O}{\zeta} = \sum_{\zeta \in O}{\zeta^2} = \pm 1$ or $0$. 
If we have $+$ we replace $O$ by the orbit containing only $1$; if we have $-$ we replace $O$ by the orbit containing $3^{\text{rd}}$ roots of unity; and if we have $0$, we drop the orbit from our sums.
\item $2 \mid\mid Y$ and $\sum_{\zeta \in O}{\zeta} = \pm 1$ and $\sum_{\zeta \in O}{\zeta^2} = \mp 1$.
If we have the top sign we replace $O$ by the orbit of primitive $6^{\text{th}}$ roots of unity; if we have the bottom sign we replace $O$ by the orbit containing $-1$.
\item $4 \mid\mid Y$ and $\sum_{\zeta \in O}{\zeta} = 0$ and $\sum_{\zeta \in O}{\zeta^2} = \pm 2$.
If we have the top sign we replace $O$ by the orbit of primitive $12^{\text{th}}$ roots of unity; if we have the bottom sign we replace $O$ by the orbit of primitive $4^{\text{th}}$ roots of unity.
\item $8 \mid Y$ and $\sum_{\zeta \in O}{\zeta} = \sum_{\zeta \in O}{\zeta^2} = 0$. We drop this orbit from our sums.
\end{enumerate}
All of these replacement orbits have size less than or equal to $|O|$ and all contain $L^{\text{th}}$ roots of unity.

{\textbf{\boldmath{Case 2: $n \mid Y$ and $b > 2 + \epsilon$ or $r_l > 1$ for some $l$.}}} Define $k = 4$ if $b > 2 + \epsilon$, and $k = p_l$ otherwise. We argue as before, and note that
\vspace{-12pt}
\begin{singlespace}
\[\sum_{\zeta \in O}{\zeta^2} = \sum_{\zeta \in O}{\zeta_k^2 \cdot \zeta^2} = \zeta_k^2 \sum_{\zeta \in O}{\zeta^2.}\]
\end{singlespace}
\noindent
Since $k > 2$ this implies $\sum_{\zeta \in O} \zeta = \sum_{\zeta \in O}{\zeta^2} = 0$. Thus we can drop $O$ from our sums.

{\textbf{\boldmath{Case 3: $n \mid Y$ and $b \leq 2 + \epsilon$ and $r_i \leq 1$ for all $i$.}}}
Note that since $u$ is odd, the squares of the primitive $u^{\text{th}}$ roots of unity are the primitive $u^{\text{th}}$ roots of unity, so
\vspace{-12pt}
\begin{singlespace}
\[\sum_{\zeta \in O}{\zeta} = \sum_{\zeta_v \in O'}{\zeta_v} \ \cdot \sum_{i, (i, u) = 1}{\zeta_u} \quad \text{and} \quad \sum_{\zeta \in O}{\zeta^2} = \sum_{\zeta_v \in O'}{\zeta_v^2} \ \cdot \sum_{i, (i, u) = 1}{\zeta_u}.\]
\end{singlespace}
\noindent
If $u$ is not squarefree then $\sum_{\zeta \in O}{\zeta} = \sum_{\zeta \in O}{\zeta^2} = 0$ so we can remove $O$ from our sum. If $u$ is a squarefree product of an even number of primes, $\sum_{\zeta \in O}{\zeta} = \sum_{\zeta_v \in O'}{\zeta_v}$ and $\sum_{\zeta \in O}{\zeta^2} = \sum_{\zeta_v \in O'}{\zeta_v^2}$, so we can replace $O$ by $O'$ (which is an orbit containing $L^{\text{th}}$ roots of unity because $v \mid n$ implies $v \mid L$). 
Lastly, if $u$ is a squarefree product of an odd number primes, $\sum_{\zeta \in O} \zeta = (\omega + \omega^2) \sum_{\zeta_v \in O'} \zeta_v$, and $\sum_{\zeta \in O} \zeta^2 = (\omega^2 + \omega^4) \sum_{\zeta_v \in O'} \zeta_v^2$ where $\omega$ is a primitive cube root of unity. Replacing $O$ by $\omega O' + \omega^2 O'$ leaves us with a Galois-invariant collection of $L^{\text{th}}$ roots of unity.
\end{proof}

\begin{defi} \label{defP}
For nonnegative integers $a, b, c, d$, with $c = \prod_{i=1}^t p_i$ squarefree, we define
\vspace{-12pt}
\begin{singlespace}
\[P(a, b, c, d) = \begin{cases} b\phi(c) + d\phi(c) + b + 2a &\text{if $c \equiv 3 \mod 4$ and $t$ is odd} \\ b\phi(c) - 2b + 2a &\text{otherwise.}\end{cases} \]
\end{singlespace}
\noindent
\end{defi}

\begin{prop}[Proposition \ref{rootsprop}]
Let $c = \prod_{i = 1}^t{p_i}$ be an odd integer divisible by $3$ where the $p_i$ are distinct primes, and let $a, b, d \in \zz_{\geq 0}$. If there are $N$ roots of unity $\theta_i$ such that $\sum{\theta_i} = -a - b\sqrt{c}$ and $\sum{\theta_i^2} = -a - d\sqrt{c}$, then $N \geq P(a, b, c, d)$.
\end{prop}
\begin{proof}
Assume to the contrary that we could write $\sum{\theta_i} = -a - b\sqrt{c}$ and $\sum{\theta_i^2} = -a - d\sqrt{c}$ with less than $P(a, b, c, d)$ roots of unity. Let $n = 2c$ if $c \equiv 1 \mod 4$ and $n = 4c$ if $c \equiv 3 \mod 4$. Then by Lemma~\ref{newrtlm}, for some $M$, we could write
\begin{equation}
\sum{\theta_i'} = M(-a - b\sqrt{c}) \quad \text{and} \quad \sum{\theta_i'^2} = M(-a - d\sqrt{c}) \label{ding}
\end{equation}
with less than $M \cdot P(a, b, c, d)$ roots of unity, where the $\theta_i'$ are $2n^{\text{th}}$ roots of unity whose multiplicities are Galois-invariant. To show we need at least $M \cdot P(a, b, c, d)$ roots of unity $\theta_i$ to write \eqref{ding}, take some orbit with primitive $Y^{\text{th}}$ roots unity. First we consider the case when $\sqrt{c} \notin \qq(\zeta_Y)$, so the orbit consists of all primitive $Y^{\text{th}}$ roots of unity. Write $Y = 2^b \prod_{i = 1}^{k}{p_i^{r_i}}$. We compute:
\begin{center}
\begin{tabular}{c|c|c|c} 
type of orbit & size & sum of orbit & sum of squares in orbit \\ 
\hline
$b \geq 3$, or $r_i > 1$ for some $i$ & $\phi(Y) \geq 0$ & $0$ & $0$ \\ 
\hline
$b = 0$, $k$ even, $r_i = 1 \ \forall i$ & $\phi(Y) \geq 1$ & $1$ & $1$ \\
\hline
$b = 0$, $k$ odd, $r_i = 1 \ \forall i$ & $\phi(Y) \geq 2$ & $-1$ & $-1$ \\
\hline
$b = 1$, $k$ even, $r_i = 1 \ \forall i$ & $\phi(Y) \geq 1$ & $-1$ & $1$ \\
\hline
$b = 1$, $k$ odd, $r_i = 1 \ \forall i$ & $\phi(Y) \geq 2$ & $1$ & $-1$ \\
\hline
$b = 2$, $k$ even, $r_i = 1 \ \forall i$ & $\phi(Y) \geq 1$ & $0$ & $-2$ \\
\hline
$b = 2$, $k$ odd, $r_i = 1 \ \forall i$ & $\phi(Y) \geq 4$ & $0$ & $2$ \\
\end{tabular}
\end{center}

To find the minimum number of roots of unity $\theta_i$ used to write \eqref{ding}, if we can write one row as the sum
of other rows in a way that requires less roots of unity, we can assume we do not use that row. Deleting rows in this way, we arrive at the following:

\begin{center}
\begin{tabular}{c|c|c|c} 
orbit & size & sum of orbit & sum of squares in orbit \\ 
\hline
$Y = 1$ & $1$ & $1$ & $1$ \\ 
\hline
$Y = 2$ & $1$ & $-1$ & $1$ \\
\hline
$Y = 3$ & $2$ & $-1$ & $-1$ \\
\hline
$Y = 6$ & $2$ & $1$ & $-1$ \\
\hline
$Y = 4$ & $2$ & $0$ & $-2$ \\
\end{tabular}
\end{center}

Assume $c \equiv 3$ mod $4$. The cases that remain are $Y = 4c, 8c$. Here there are two Galois orbits $O_1$ and $O_2$ of $Y^{\text{th}}$ roots of unity. We compute $\sum_{\zeta \in O_1}{\zeta} + \sum_{\zeta \in O_2}{\zeta} = 0$
and 
\vspace{-12pt}
\begin{singlespace}
\[\sum_{\zeta \in O_1}{\zeta^2} + \sum_{\zeta \in O_2}{\zeta^2} = \begin{cases} 0 & \text{if $Y = 8c$;} \\ 2 \cdot (-1)^{t + 1} & \text{if $Y = 4c$.}\end{cases}\]
\end{singlespace}
\noindent
When we take the differences, this is the product of Gauss sums: 
\vspace{-12pt}
\begin{singlespace}
\[\sum_{\zeta \in O_1}{\zeta} - \sum_{\zeta \in O_2}{\zeta} = \begin{cases}0 & \text{if $Y = 8c$} \\ 2\sqrt{c} & \text{if $Y = 4c$}\end{cases} \qquad \text{and} \qquad \sum_{\zeta \in O_1}{\zeta^2} - \sum_{\zeta \in O_2}{\zeta^2} = \begin{cases}\pm 4\sqrt{c} & \text{if $Y = 8c$} \\ 0 & \text{if $Y = 4c$}\end{cases}\]
\end{singlespace}
\noindent
Solving these systems of equations, we have:

\begin{center}
\begin{tabular}{c|c|c|c} 
orbit & size & sum of orbit & sum of squares in orbit \\ 
\hline
$Y = 4c$ & $\phi(c)$ & $\sqrt{c}$ & $ (-1)^{t + 1}$ \\ 
\hline
$Y = 4c$ & $\phi(c)$ & $-\sqrt{c}$ & $ (-1)^{t + 1}$ \\
\hline
$Y = 8c$ & $2\phi(c)$ & $0$ & $2\sqrt{c}$ \\
\hline
$Y = 8c$ & $2\phi(c)$ & $0$ & $-2\sqrt{c}$
\end{tabular}
\end{center}

When $t$ is odd, define the linear function 
$f(x + y\sqrt{c}, z + w\sqrt{c}) = y\phi(c) + w\phi(c) + x + z + y$.
If there were orbits $O_1, \ldots, O_m$ with $\sum_i{\sum_{\zeta \in O_i}{\zeta}} = M(-a - b\sqrt{c})$ and $\sum_i{\sum_{\zeta \in O_i}{\zeta^2}} = M(-a - d\sqrt{c})$,
\vspace{-12pt}
\begin{singlespace}
\[\sum_i{|O_i| \cdot \frac{1}{|O_i|}f(\sum_{\zeta \in O_i}{\zeta}}, \sum_{\zeta \in O_i}{\zeta^2}) = f(\sum_i{\sum_{\zeta \in O_i}{\zeta}}, \sum_i{\sum_{\zeta \in O_i}{\zeta^2}}) = -M(b\phi(c) + d\phi(c) + 2a + b).\]
\end{singlespace}
\noindent
Since $-1 \leq \frac{1}{|O_i|}f(\sum_{\zeta \in O_i}{\zeta}, \sum_{\zeta \in O_i}{\zeta^2})$ for all $i$,
\begin{gather*}
\sum_i{|O_i| \cdot (-1)} \leq \sum_i{|O_i| \cdot \frac{1}{|O_i|}f(\sum_{\zeta \in O_i}{\zeta}}, \sum_{\zeta \in O_i}{\zeta^2}) = -M(b\phi(c) + d\phi(c) + 2a + b) \\
\Rightarrow \quad \sum_i{|O_i|} \geq M(b\phi(c) + d\phi(c) + 2a + b).
\end{gather*}

Next, assume $c \equiv 1$ mod $4$. It remains to consider $Y = c, 2c, 4c$. We make a similar calculation as above, taking the top sign when $c \equiv 5$ mod $8$ and the lower sign when $c \equiv 1$ mod $8$, to produce the table:

\begin{center}
\begin{tabular}{c|c|c|c} 
orbit & size & sum of orbit & sum of squares in orbit \\ 
\hline
$Y = c$ & $\frac{1}{2}\phi(c)$ & $\frac{1}{2}\left((-1)^t + \sqrt{c}\right)$ & $\frac{1}{2}\left((-1)^t \mp \sqrt{c}\right)$ \\ 
\hline
$Y = c$ & $\frac{1}{2}\phi(c)$ & $\frac{1}{2}\left((-1)^t - \sqrt{c}\right)$ & $\frac{1}{2}\left((-1)^t \pm \sqrt{c}\right)$ \\
\hline
$Y = 2c$ & $\frac{1}{2}\phi(c)$ & $\frac{1}{2}\left((-1)^{t + 1} \pm \sqrt{c}\right)$ & $\frac{1}{2}\left((-1)^t + \sqrt{c}\right)$ \\
\hline
$Y = 2c$ & $\frac{1}{2}\phi(c)$ & $\frac{1}{2}\left((-1)^{t + 1} \mp \sqrt{c}\right)$ & $\frac{1}{2}\left((-1)^t - \sqrt{c}\right)$\\
\hline
$Y = 4c$ & $\phi(c)$ & $0$ & $(-1)^{t+1} + \sqrt{c}$ \\
\hline
$Y = 4c$ & $\phi(c)$ & $0$ & $(-1)^{t+1} - \sqrt{c}$
\end{tabular}
\end{center}

Define the function $f(x + y\sqrt{c}, z + w\sqrt{c}) = y\phi(c) + x + z - 2y$. Assume there were orbits $O_1, \ldots, O_n$ such that $\sum_i{\sum_{\zeta \in O_i}{\zeta}} = M(-a - b\sqrt{c})$ and $\sum_i{\sum_{\zeta \in O_i}{\zeta^2}} = M(-a - d\sqrt{c})$. In all cases, we find 
\vspace{-12pt}
\begin{singlespace}
\[-1 \leq \frac{1}{|O|}f\left(\sum_{\zeta \in O}{\zeta}, \sum_{\zeta \in O}{\zeta^2}\right).\]
\end{singlespace}
\noindent
So by the same argument as above, we have $\sum_i{|O_i|} \geq M(b\phi(c) + 2a - 2b)$.
\end{proof}

\vspace{.1in}
\noindent
\textbf{Acknowledgements.}
I would like to thank Victor Ostrik for suggesting this problem, providing me with useful papers to read, and meeting with me weekly to discuss ideas. I would also like to thank my brother, Eric Larson, for teaching me about the theory of cyclotomic fields (such as some of the methods used in section \ref{rootssection}) and proof-reading the paper.

\nocite{*}
\bibliography{FCrefs}{}
\bibliographystyle{plain}

\end{document}